\documentclass{amsart}
\usepackage{amssymb}
\usepackage{hyperref}
\usepackage{url}
\usepackage[all]{xy}
\usepackage{tikz}
\usetikzlibrary{matrix,arrows,decorations.pathmorphing}

\newtheorem{thm}{Theorem}
\newtheorem{introthm}{Theorem}

\newtheorem{lem}[thm]{Lemma}
\newtheorem{cor}[thm]{Corollary}

\newtheorem{prop}[thm]{Proposition}

\theoremstyle{definition}
\newtheorem{defn}[thm]{Definition}

\newtheorem{rem}[thm]{Remark}

\numberwithin{thm}{section}

\newfont{\cyrr}{wncyr10}
\def\Sh{\mbox{\cyrr Sh}}

\def\Z{\mathbf{Z}}
\def\Q{\mathbf{Q}}
\def\F{\mathbf{F}}
\def\R{\mathbf{R}}
\def\C{\mathbf{C}}

\def\Fp{\F_p}
\def\Ftwo{\F_2}

\def\bA{\mathbf{A}}

\def\Xset{\mathcal{C}}

\def\Xset{\mathcal{C}}

\def\A{\mathcal{A}}
\def\O{\mathcal{O}}

\def\cP{\mathcal{P}}

\def\cP{\mathcal{P}}
\def\H{\mathcal{H}}

\def\l{\mathfrak{q}}

\def\p{\mathfrak{p}}
\def\q{\mathfrak{q}}

\def\k{\Bbbk}

\def\Aut{\mathrm{Aut}}
\def\Hom{\mathrm{Hom}}
\def\Gal{\mathrm{Gal}}

\def\ord{\mathrm{ord}}

\def\Aut{\mathrm{Aut}}
\def\Sel{\mathrm{Sel}}
\def\End{\mathrm{End}}

\def\Frob{\mathrm{Frob}}

\def\ur{\mathrm{ur}}

\def\loc{\mathrm{loc}}

\def\dimtwo{\dim_{\Ftwo}}

\def\ram{\mathrm{ram}}

\def\Pic{\mathrm{Pic}}
\def\iK{\bA_K^\times}

\def\N{\mathbf{N}}

\def\too{\longrightarrow}

\def\dirsum#1{\underset{#1}{\textstyle\bigoplus}}

\title[Selmer Ranks of twists of hyperelliptic curves]
   {Selmer Ranks of twists of hyperelliptic \\ curves and superelliptic curves$^\star$}
\thanks{$^\star$This material is based upon work supported by the National Science Foundation under grant DMS-1065904.}
\author{Myungjun Yu}
\address{Department of Mathematics,
UC Irvine,
Irvine, CA 92697,
USA}
\email{\href{mailto:myungjuy@math.uci.edu}{myungjuy@math.uci.edu}}

\begin{document}

\begin{abstract}
We study the variation of Selmer ranks of Jacobians of twists of hyperelliptic curves and superelliptic curves. We find sufficient conditions for such curves to have infinitely many twists whose Jacobians have Selmer ranks equal to $r$, for any given nonnegative integer $r$. This generalizes earlier results of Mazur-Rubin on elliptic curves.

\end{abstract}

\maketitle

\section*{Introduction}
Let $E$ be an elliptic curve over $\Q$. Goldfeld's conjecture predicts that
$50\%$ of elliptic curves have rank $0$ and $50\%$ have rank $1$ in the family of quadratic twists of $E$. Let $E^\chi$ denote the quadratic twist of $E$ by a quadratic character $\chi$. We write $\Sel_2(E/\Q)$ and $\Sh_{E/\Q}$ to denote the $2$-Selmer group and the Shafarevich-Tate group of $E/\Q$, respectively. Then the following sequence is exact for all $\chi$:
$$
0 \to E^\chi(\Q)/2E^\chi(\Q) \to \Sel_2(E^\chi/\Q) \to \Sh_{E^\chi/\Q}[2] \to 0.
$$
Thereby we deduce that
$$
\dimtwo(\Sel_2(E^\chi/Q)) = \mathrm{rk}(E^\chi(\Q)) + \dimtwo(E^\chi(\Q)[2]) + \dimtwo(\Sh_{E^\chi/\Q}[2]),
$$
where $\mathrm{rk}(E^\chi(\Q))$ is the Mordell-Weil rank (over $\Q$) of $E^\chi$.
If we assume the Shafarevich-Tate Conjecture along with Goldfeld's conjecture, so $\dimtwo(\Sh_{E^\chi/\Q}[2])$ is even by Cassels' pairing, we have (varying $\chi$)
\begin{equation}
\label{e50percent}
\begin{aligned}
\bullet \text{ }50\%& \text{ of } E^\chi \text{ have even $2$-Selmer ranks and }\\
\bullet \text{ }50\%& \text{ of } E^\chi \text{ have odd $2$-Selmer ranks}
\end{aligned}
\end{equation}
in the family of quadratic twists of $E/\Q$, since $\dimtwo(E^\chi(\Q)[2])$ are the same for all $\chi$. This turns out to be true even if we assume neither of those big conjectures (see \cite[Corollary 7.10]{KMR}).

Let $C_{2,f}/\Q$ be a hyperelliptic curve whose affine model is
$$
y^2 = f(x),
$$
where $f$ is a (separable) polynomial defined over $\Q$, and $\deg(f)$ is odd. Then it is interesting to know if we can show that \eqref{e50percent} holds when $E$ is replaced by the Jacobian of $C_{2,f}$ (denoted by $J$). In other words, is it true that
\begin{equation}
\label{50percent}
\begin{aligned}
\bullet \text{ }50\%& \text{ of } J^\chi \text{ have even $2$-Selmer ranks, and }\\
\bullet \text{ }50\%& \text{ of } J^\chi \text{ have odd $2$-Selmer ranks}
\end{aligned}
\end{equation}
in the family of quadratic twists of $J/\Q$? It turns out to be true if $\deg(f) \equiv 3 \text{ (mod }4)$, and false if $\deg(f) \equiv 1 \text{ (mod }4)$ (see Corollary \ref{application}, Proposition \ref{heng}, and Proposition \ref{heng2}). It is even possible that $\dimtwo(\Sel_2(J^\chi/\Q))$ has constant parity for all $\chi$ when $\deg(f) \equiv 1 \text{ (mod }4)$. We display such examples in Section $8$. In fact, we prove (Theorem \ref{paritydensity}):
\begin{introthm}
\label{kmrkmr}
For all sufficiently large $X$,
$$
\frac{|\{\chi \in \Xset(K, X):d_2(\Sel_2(J^\chi/K)) \text{ is even }    \}|}{|\Xset(K, X)|} = \frac{1 + \delta}{2},
$$
where $\Xset(K,X)$ and $\delta$ are defined in Definition \ref{7.2} and Definition \ref{7.1}, respectively.
\end{introthm}
This theorem generalizes an earlier result of Klagsbrun, Mazur, and Rubin on elliptic curves (\cite[Theorem 7.6]{KMR}) to Jacobians of hyperelliptic curves.

Aside from the parity distribution, two questions we could ask for the family of quadratic twists of $J$ over a number field $K$ are, ``how many integers does the set $A_J:=\{\dimtwo(\Sel_2(J^\chi/K))|\chi \text{ is a quadratic character}\}$ cover," and, ``under what conditions does $A_J = \Z_{\ge0}$ happen?" This question was inspired by Mazur-Rubin's work (\cite[Theorem 1.6]{hilbert}) on elliptic curves as follows:
\begin{introthm}\cite[Theorem 1.6]{hilbert}
\label{thm1}
Suppose that $K$ is a number field, and $E$ is an elliptic curve over $K$
such that $\Gal(K(E[2])/K) \cong S_3$.  Let $\Delta_E$ be the discriminant
of some model of $E$, and suppose further that $K$ has a place $v_0$ satisfying
one of the following conditions:
\begin{itemize}
\item
$v_0$ is real and $(\Delta_E)_{v_0} < 0$,
or
\item
$v_0 \nmid 2\infty$, $E$ has multiplicative reduction at $v_0$, and
$\ord_{v_0}(\Delta_E)$ is odd.
\end{itemize}
Then for every $r \ge 0$, there are infinitely many quadratic twists $E'/K$ of $E$ such that $\dim_{\Ftwo}(\Sel_2(E'/K)) = r$.
\end{introthm}
 We study the generalization of Theorem \ref{thm1} to Jacobians of hyperelliptic and superelliptic curves in this paper. We write $C_{p,f}$ to denote the superelliptic curve (Definition \ref{superdef}), which is the smooth projective curve whose affine model is given by $y^p = f(x)$, where $f(x)$ is a separable polynomial of degree $\ge 2$. In particular, when $p=2$, we call $C_{2,f}$ a hyperelliptic curve. We denote the Jacobian of a (smooth projective) curve $C$ by $J(C)$. We prove (Corollary \ref{hec}):
\begin{introthm}
\label{assumption}
Suppose that $K$ is a number field and $f \in K[x]$ is a separable polynomial. Let $n = \deg(f)$ and suppose that $n \equiv 3 (\textrm{mod }4)$ and $\Gal(f) \cong S_n$ or $A_n$. Suppose further that $K$ has a real embedding. Then for every $r \geqq 0$, the Jacobian $J$ of $C_{2,f}$ has infinitely many quadratic twists $J^\chi$ such that   $\dim_{\Ftwo}(\Sel_2(J^\chi/K))$ = $r$.
\end{introthm}
Our technique relies on the fact that there is a canonical $G_{K}$-module isomorphism between $J(C_{2,f})[2]$ and $J(C_{2,df})[2]$, so that we can compare two Selmer groups
\begin{align*}
\Sel_2(J(C_{2,f})/K) &\subset H^1(K, J(C_{2,f})[2]) \text{ and }  \\
\Sel_2(J(C_{2, df})/K) &\subset H^1(K, J(C_{2, df})[2]) \cong H^1(K, J(C_{2,f})[2])
\end{align*}
in $H^1(K, J(C_{2,f})[2])$.
We work more generally with $\pi$-Selmer groups ($\pi := 1-\zeta_p$, where $\zeta_p$ is a primitive $p$-th root of unity) of superelliptic curves $C_{p,f}$ for arbitrary prime numbers $p$. Note that if $p=2$, then $\pi = 1-\zeta_2=2$, so $\pi$-Selmer group (which will be defined in Section 3) is actually a $2$-Selmer group of a hyperelliptic curve $C_{2,f}$ in this case. As in the hyperelliptic curve case, it is essential to compare $\Sel_\pi(J(C_{p,f})/K)$ and $\Sel_\pi(J(C_{p,df})/K)$ in $H^1(K, J(C_{p,f})[\pi])$, which is possible because there is a $G_{K}$-module isomorphism between $J(C_{p,f})[\pi]$ and $J(C_{p,df})[\pi]$ as will be seen in Proposition \ref{canonicaliso}.
We get a similar result when $p \geqq 3$, as follows (Theorem \ref{sec}).
\begin{introthm}
\label{assumption2}
Suppose that $K$ is a number field containing $\zeta_p$, and $f \in K[x]$ is a separable polynomial. Let $n = \deg(f)$ and suppose that $p \nmid n$ is an odd prime and $\Gal(f) \cong S_n$. Then for every $r$ $\geqq 0$, the superelliptic curve $C_{p, f}$ has infinitely many $p$-twists $C_{p, df}$ where $d \in K^\times/(K^\times)^p$ such that $\dim_{\Fp}(\Sel_{\pi}(J(C_{p,df})/K))$ = $r$.
\end{introthm}

In the elliptic curve case, Kramer showed that (see \cite[Theorem 1]{kramer} and \cite[Theorem 2.8]{hilbert}) there is a (parity) relation between two Selmer groups $\Sel_2(E/K)$ and $\Sel_2(E^\chi/K)$ as follows:
$$
\mathrm{dim}_{\Ftwo}(\Sel_2(E/K)) - \mathrm{dim}_{\Ftwo}(\Sel_2(E^\chi/K)) \equiv
\displaystyle{\sum_{v}}{h_v(\chi_v)} (\textrm{mod } 2),
$$
where $\chi_v$ is the restriction of $\chi$ to $G_{K_v}$, and $h_v(\chi_v)$ is defined locally for every place $v$ (Definition \ref{hv}). In this paper, with the aid of \cite[Theorem 3.9]{KMR} we extend this result to Jacobians of hyperelliptic curves (Theorem \ref{parity}). In fact, this generalization plays an important role in proving Theorem \ref{kmrkmr}.

We briefly sketch the main idea of the proofs of Theorem \ref{assumption} and Theorem \ref{assumption2} here. Since a Selmer group is determined by its local conditions, if we find an element $d \in K^\times/(K^\times)^p$ such that local conditions defining $\Sel_\pi(J(C_{p,f})/K)$ and $\Sel_\pi(J(C_{p, df})/K)$ are the same everywhere except at only one place $v$, then we may expect $\Sel_\pi(J(C_{p,f})/K)$ and $\Sel_\pi(J(C_{p, df})/K)$ are ``almost" the same. By choosing an appropriate $d$, we increase or decrease the $\pi$-Selmer ranks by $1$ under the assumptions of Theorem \ref{assumption2}. But in the hyperelliptic curve case, it turns out it is difficult to increase or decrease $2$-Selmer ranks by $1$ by twisting only with this idea. But under sufficiently good conditions (see Theorem \ref{selmerrankpm2}), it is possible to increase or decrease $2$-Selmer ranks by $2$ by twisting. To prove this, we use the structure of metabolic spaces and the quadratic forms arising from the Heisenberg groups. Then with the aid of Theorem \ref{kmrkmr}, we prove Theorem \ref{assumption}.

The layout of the paper is as follows. In Section 1, we define a $p$-twist $J^\chi$ of the Jacobian $J$ of a superelliptic curve $C_{p,f}$. In Section 2, we assume $C_{p,f}$ is defined over a local field $K_v$. We compare local conditions (the Kummer images arising from $\pi: J \to J$, and $\pi: J^\chi \to J^\chi$), which are in fact the defining local conditions of $\pi$-Selmer groups of $J$ and $J^\chi$. In Section 3, we define various Selmer groups. In Section 4, Theorem \ref{assumption2} will be proved. Section 5 defines the quadratic form arising from Heisenberg group that we will need in the sequel. Section 6 and Section 7 will be devoted to prove Theorem \ref{kmrkmr} and Theorem \ref{assumption}, which shows that \eqref{50percent} is true if $\deg(f) \equiv 3 \text{ }(\text{mod }4)$. Section 8 contains an explicit example of the Jacobian of $C_{2,f}$ ($\deg(f) \equiv 1 \text{ }(\text{mod }4))$ , whose quadratic twists have constant 2-Selmer rank parity, so \eqref{50percent} is not satisfied.


\section{Superelliptic curves and twists}
In this section, we fix a prime $p$ and a field $L$ containing $\zeta_p$, where $\zeta_p$ is a primitive $p$-th root of unity. We assume that $L$ has characteristic $0$. Consider a plane curve
$$
S : y^p = f(x),
$$
where $f$ is a separable polynomial over $L$. If $p\nmid \mathrm{deg}(f)$, then $S$ is normalized to a smooth projective curve, with one rational point at infinity that we will denote by $\infty$. See \cite[\S1]{normalization} for the justification. Note that this holds for all $f$ such that $p\nmid \deg(f)$ although \cite{normalization} treats only the case when $\deg(f) > p$ for another reason.

\begin{defn}
\label{superdef}
A superelliptic curve $C_{p,f}$ over $L$ is a smooth projective curve in the projective space $\mathbb{P}^2$ whose affine model is
$$y^p = f(x),$$
where $f$ is a separable polynomial (not necessarily monic) defined over $K$ such that $p \nmid \deg(f)$. When $p=2$, we call it a hyperelliptic curve.
\end{defn}
Suppose that $\alpha_1, \alpha_2, \cdots \alpha_n$ are the roots of $f$, where $n$ denotes the degree of $f$.
For every smooth projective curve $C$, let $J(C)$ denote the Jacobian of $C$. Then $J(C_{p,f})$ is generated by divisor classes of the form $[P - \infty],$ where $P$ is a point of $y^p = f(x)$.

Fix a primitive $p$-th root of unity $\zeta_p$, and let $\pi:= 1-\zeta_p$. Then we have $\pi^{p-1} = pu$ for a unit $u \in \Z$[$\zeta_p$].  We often write $J$ instead of $J(C_{p,f})$ when it is obvious in the context.

\begin{rem}
\label{action}
Let $C_{p,f}$ be a superelliptic curve definied over $L$ of characteristic $0$. Note that $J(C_{p,f})$ has a natural $\Z[\zeta_p]$-action induced by $\zeta_p(\alpha, \beta) = (\alpha, \zeta_p\beta),$ where $(\alpha, \beta)$ is a point of $y^p = f(x)$. In other words, there is a natural map
$$
\boldsymbol{\mu}_p \to \mathrm{Aut}(J(C_{p,f})),
$$
where $\boldsymbol{\mu}_p$ is the multiplicative group of $p$-th roots of unity. Let $\alpha_1,  \alpha_2, \cdots, \alpha_n$ be the roots of $f(x)$. Then
$$
[(\alpha_1,0)-\infty], [(\alpha_2,0)-\infty], \cdots, [(\alpha_{n-1},0)-\infty]
$$
form a basis of $J[\pi]$ by \cite[Proposition 3.2]{pi-torsion}, where $J[\pi]$ denotes the $\Fp$-vector space of the $\pi$-torsion points of $J(C_{p,f})$. Note that
$$
[(\alpha_n,0)-\infty] = -[(\alpha_1,0)-\infty] - [(\alpha_2,0)-\infty] - \cdots - [(\alpha_{n-1},0)-\infty].
$$
\end{rem}

\begin{rem}
\label{snaction}
The symmetric group $S_n$ naturally acts on the set of the roots of $f$. This induces an action of $S_n$ on $J[\pi]$. Note that $\Gal(f) \subset S_n$ acts on $J[\pi]$ naturally too.
\end{rem}

\begin{defn}
Let $L$ be a field of characteristic $0$, and $\zeta_p \in L$ . We write
$$
\Xset^p(L) := \Hom(G_L, \boldsymbol{\mu}_p).
$$
If $L$ is a local field, we often identify $\Xset^p(L)$ with $\Hom(L^\times, \boldsymbol{\mu}_p)$ via the local reciprocity map, and let $\Xset^p_\ram(L) \subset \Xset^p(L)$ be the subgroup of ramified characters in $\Xset^p(L)$. Then $\chi \in \Xset^p_\ram(L)$ if and only if $\chi(\O_L^\times) \neq 1$, where $\O_L^\times$ is the unit group of the integer ring of $L$, by local class field theory.
\end{defn}

\begin{defn}
\label{definitionofquadratictwist}
For $\chi \in \Xset^p(L)$, we say that $J^\chi$ is the $p$-twist of $J$ by $\chi$ if $J^\chi$ is a superelliptic curve whose affine model is
$$
y^2 = d^{-1}f(x),
$$
where $d$ is the preimage of $\chi$ in the Kummer map
$$
L^\times/(L^\times)^p \cong \Hom(G_L, \boldsymbol{\mu}_p);
$$
i.e., $J^\chi = J(C_{p, d^{-1}f})$.
\end{defn}

\begin{rem}
\label{caniso}
For any superelliptic curve $C_{p,f}$ defined over $K$ of characteristic $0$, an isomorphism over $\overline{K}$ between two affine curves
$$
\{y^p = d^{-1}f(x)\} \cong \{y^p = f(x)\}
$$
is given by sending $(a, b)$ to $(a, \sqrt[p]{d}b)$ where $\sqrt[p]{d}$ is a choice of $p$-th root of $d$. The following proposition is an immediate consequence of Remark \ref{action}.
\end{rem}

\begin{prop}
\label{canonicaliso}
Suppose that a superelliptic curve $C_{p,f}$ is defined over a field $L$ of characteristic $0$, and $\chi \in \Xset^p(L)$. Then
$$
J[\pi] \cong J^\chi[\pi].
$$
In particular, $J(L)[\pi] \cong J^\chi(L)[\pi]$.
\end{prop}

\begin{rem}
Let $J/L$ be the Jacobian of a superelliptic curve $C_{p,f}$. We denote the set of twists of $J/L$ by $\mathrm{Twist}(J/L)$. It is well-known (for example, see Proposition 5 in \cite[Chapter3 \S1]{serre}) that there is a bijection  $$\mathrm{Twist}(J/L) \to H^1(G_L, \mathrm{Aut}(J)).$$
It maps $\phi : A' \to J$ to $\xi: G_L \to \mathrm{Aut}(J),$ where $\xi_\sigma = \phi^\sigma \circ \phi^{-1}$.

Then we have a composition of maps
\begin{equation}
\label{quadtwistmap}
\Hom(G_L, \boldsymbol{\mu}_p) \to H^1(G_L, \mathrm{Aut}(J)) \to \mathrm{Twist}(J/L),
\end{equation}
where the first map is given by the map $\boldsymbol{\mu}_p \to \mathrm{Aut}(J)$ in Remark \ref{action}. The last map is the bijection given above. A $p$-twist of $J$ by $\chi \in \Xset^p(K)$ is then, given by the image of $\chi$ in \eqref{quadtwistmap}.
\end{rem}

\section{Comparison of local conditions}
In this section, let $C_{p,f}$ denote a superelliptic curve defined over a local field $K_v$ containing the $p$-th roots of unity, where $n:= \deg(f)$, so that $\Gal(f) \subset S_n$. For the rest of the paper, a local field is either an archimedean field, or a finite extension of $\Q_{\ell}$ for some prime number $\ell$. If $K_v$ is an archimedean field, we write $v|\infty$, and if the residue characteristic of $K_v$ is $p$, we write $v|p$. Let $1_v$ denote the trivial character in $\Xset^p(K_v)$.
We denote the Jacobian of $C_{p,f}$ simply by $J$. The local conditions in Definition \ref{hv} are in fact those of the $\pi$-Selmer group (to be defined later) of the Jacobian of a superelliptic curve defined over a number field. In this section, we list several lemmas to be used in the sequel.

\begin{defn}
Let $V$ be a finite dimensional vector space over $\Fp$. We write $d_p(V)$ for the dimension of $V$ over $\Fp$.
\end{defn}

\begin{defn}
\label{hv}
For $\chi \in \Xset^p(K_v)$, define
$$
\alpha_v(\chi) := \mathrm{Im}( J^{\chi}(K_v)/\pi J^{\chi}(K_v) \to H^1(K_v, J^{\chi}[\pi]) \cong H^1(K_v, J[\pi])),
$$
where the first map is given by the Kummer map. Define
$$
h_v(\chi):= d_p(\alpha_v(1_v)/(\alpha_v(1_v) \cap \alpha_v(\chi)).
$$
\end{defn}
The vector space $\alpha_v(\chi)$ and the invariant $h_v(\chi)$ certainly depend on $C_{p,f}$ and the field $K_v$ over which $C_{p,f}$ is defined, but we suppress them from the notation.

Let $\lambda : J \to \hat{J}$ be the canonical principal polarization. Then $J[\pi]$ is self-dual; i.e., $\lambda^{-1}(\hat{J}[\hat{\pi}]) = J[\pi]$, where $\hat{\pi}$ is the dual isogeny of $\pi$(see \cite[Proposition 3.1]{pi-torsion}). Let $\langle  \text{ , } \rangle_\pi$ denote the Cartier pairing (for example, see Section $1$ in \cite{cartier}) for $\pi : J \to J$ (multiplication by $\pi$).

\begin{defn}
\label{pairing}
Define a pairing
$$
e_\pi : J[\pi] \times J[\pi] \to \boldsymbol{\mu}_p
$$
by sending $(x,y)$ to $\langle  {x, \lambda(y)} \rangle_\pi$.
\end{defn}

\begin{rem}
By properties of the Cartier pairing, the pairing $e_\pi$ is bilinear, nondegenerate, and $G_K$-equivariant. In fact, $e_\pi$ can be defined more concretely as follows. If $e_p$ is the Weil pairing of $J[p]$ associated to the canonical principal polarization $\lambda$,
$$
e_\pi (a, b) := e_p (a, (\pi^{p-2})^{-1}(b)),
$$
which can be proved by using functoriality in \cite[Corollary 1.3(ii)]{cartier}.
If $p=2$, then $e_\pi$ is nothing but the Weil pairing of $J[2]$ associated to the canonical principal polarization.
\end{rem}
The following theorem follows from Tate's local duality.
\begin{thm}
\label{tlp}
Tate's local duality and the paring $e_\pi$ give a nondegenerate pairing
\begin{equation}
\label{tld}
\langle  \text{ , } \rangle_v : H^1(K_v, J[\pi]) \times H^1(K_v, J[\pi]) \too H^2(K_v, \zeta_p) = {\Fp}.
\end{equation}
If $p=2$, then the pairing is symmetric.
\end{thm}

 \begin{proof}
For example, see \cite[Theorem 7.2.6]{cohomology}. The last assertion holds because the pairings given by cup product and the Weil pairings are alternating.
 \end{proof}

\begin{lem}
\label{selfdual}
Let $C_{p,f}$ be a superelliptic curve over a local field $K_v$ containing $\zeta_p$. Then the image of the Kummer map
 $$
 J(K_v)/\pi J(K_v) \to H^1(K_v, J[\pi])
 $$
 is its own orthogonal complement under pairing \eqref{tld}.
\end{lem}

\begin{proof}
We prove it by applying the well-known fact that the image of the Kummer map
$$
J(K_v)/p J(K_v) \rightarrow H^1(K_v, J[p])
$$
is self-dual in Tate's local duality for $H^1(K_v, J[p])$ and diagram chasing. By applying long exact sequences to the following diagrams
$$
\xymatrix{
0 \ar[r] & J[\pi] \ar[r] \ar@{^{(}->}[d] & J \ar^\pi[r] \ar@{=}[d] & J  \ar[r] \ar^{p/\pi}[d]& 0 & 0 \ar[r] & J[\pi] \ar[r] & J \ar^\pi[r] & J \ar[r] &0\\
0 \ar[r] & J[p] \ar[r] & J \ar^p[r] & J \ar[r] & 0 & 0 \ar[r] & J[p] \ar_{\pi^{p-2}}[u] \ar[r] & J \ar^p[r] \ar_{\pi^{p-2}}[u] & J \ar[r] \ar_{u}[u] &0
}
$$
we get the following commutative diagrams with exact rows
$$
\xymatrix{
0 \ar[r] &  J(K_v)/\pi J(K_v) \ar^\phi[r] \ar^{p/\pi}[d]& H^1(K_v, J[\pi]) \ar^\lambda[r] \ar^s[d] & H^1(K_v, J)[\pi] \ar[r] \ar^{s'}[d]& 0 \\
0 \ar[r] & J(K_v)/p J(K_v) \ar^{\phi'}[r] &   H^1(K_v, J[p]) \ar^{\lambda'}[r] & H^1(K_v, J)[p] \ar[r] & 0
}
$$
$$
\xymatrix{
J(K_v)/\pi J(K_v) \ar^\phi[r] & H^1(K_v, J[\pi])  \\
J(K_v)/p J(K_v) \ar^{\phi'}[r] \ar_{u}[u] & H^1(K_v, J[p]), \ar_{\pi^{p-2}}[u]
}
$$
where $\phi$ and $\phi'$ are the Kummer maps and $s$, $s'$ are natural maps.
Let $\mathrm{Im} (\phi)$ denote the set of the image of the map $\phi$. We want to show that
$\mathrm{Im} (\phi) = \mathrm{Im}(\phi)^{\perp}$. First we show that  $\mathrm{Im} (\phi) \subset \mathrm{Im}(\phi)^\perp$; i.e., $\mathrm{Im} (\phi) \perp \mathrm{Im} (\phi)$.
We have $\mathrm{Im}(s \circ \phi) \subset \mathrm{Im} (\phi')$
and $\mathrm{Im}(\pi^{p-2} \circ \phi') = \mathrm{Im}(\phi \circ u) = \mathrm{Im} (\phi)$ since the map $u$ is surjective. By the fact that $\mathrm{Im}( \phi') \perp \mathrm{Im}(\phi')$, we obtain $s^{-1}(\mathrm{Im}(\phi')) \perp \mathrm{Im}(\pi^{p-2} \circ (\phi'))$ because the following diagram commutes
$$
\begin{tikzpicture}
\matrix(m)[matrix of math nodes,
row sep=2.6em, column sep=2.8em,
text height=1.5ex, text depth=0.25ex]
{H^1(K_v, J[\pi])&\times&H^1(K_v, J[\pi])&H^2(K_v, \boldsymbol{\mu}_p)\\
H^1(K_v, J[p])&\times&H^1(K_v, J[p])&H^2(K_v, \boldsymbol{\mu}_p).\\};
\path[->,font=\scriptsize,>=angle 90]
 (m-1-1) edge node[auto] {$s$} (m-2-1)
 (m-1-3) edge node[auto] {$\cup$} (m-1-4)
(m-2-3)  edge node[right] {$\pi^{p-2}$} (m-1-3)
(m-2-3) edge node[auto] {$\cup$} (m-2-4);
\draw[double distance = 2.5pt]
(m-1-4) -- (m-2-4);
\end{tikzpicture}
$$
Therefore $\mathrm{Im}(\phi) \perp \mathrm{Im} (\phi)$, or equivalently $\mathrm{Im}(\phi) \subset \mathrm{Im}(\phi)^\perp$. Next we show  $\mathrm{Im}(\phi)^\perp \subset \mathrm{Im} (\phi)$. Let $b \in H^1(K_v, J[\pi])$ be an orthogonal element to $\mathrm{Im}(\phi)$. Then
\begin{align*}
s(b) \perp (\pi^{p-2})^{-1} ( \mathrm{Im} \phi) & \Rightarrow s(b) \perp \mathrm{Im} (\phi') \\
& \Rightarrow s(b) \in \mathrm{Im} (\phi ') \\
& \Rightarrow s(b) \in \ker(\lambda ') \\
& \Rightarrow 0 = \lambda' \circ s(b) = s' \circ \lambda(b) \\
& \Rightarrow \lambda (b) = 0 \\
& \Rightarrow b \in \ker(\lambda) = \mathrm{Im}(\phi).
\end{align*}
Hence we are done.
\end{proof}

\begin{rem}
\label{twistselfdual}
By Lemma \ref{selfdual} applied to $J^\chi$, for $\chi \in \Xset^p(K_v)$, one can see that $\alpha_v(\chi)$ is its own orthogonal complement in \eqref{tld}. Since the pairing \eqref{tld} is non-degenerate, we have
$$
d_p(H^1(K_v, J[\pi])) = 2d_p(\alpha_v(\chi)).
$$
\end{rem}

\begin{lem}
\label{hardtofindref}
Suppose that $v \nmid p\infty$ and $J/K_v$ has good reduction. Then
$$
\alpha_v(1_v) \cong J[\pi]/(\Frob_v-1)J[\pi],
$$
where the isomorphism is given by evaluating cocycles at a Frobenius automorphism $\Frob_v$.
\end{lem}

\begin{proof}
Under the assumption, we have an exact sequence
$$
\xymatrix@R=5pt@C=15pt{
0 \ar[r] & J[\pi] \ar[r] & J(K_v^\ur) \ar^-{\pi}[r] & J(K_v^\ur) \ar[r] & 0,
}
$$
where $K_v^{\ur}$ is the maximal unramified extension of $K_v$.
Taking the long exact sequence, we get an exact sequence
$$
0 \too J(K_v)/\pi J(K_v) \too H^1(K_v^{\ur}/K_v, J[\pi]) \too H^1(K_v^{\ur}/K_v, J(K_v^{\ur}))[\pi] \too 0.
$$
It is well-known that $H^1(K_v^{\ur}/K_v, J(K_v^{\ur})) = 0$ (e.g., see \cite[proposition 1]{mcc}). Hence
$$
\alpha_v(1_v) \cong J(K_v)/\pi J(K_v) \cong H^1(K_v^{\ur}/K_v, J[\pi]) \cong J[\pi]/(\Frob_v-1)J[\pi],
$$
as wanted (for the last isomorphism, see \cite[Lemma B.2.8]{EulerSystems}).
\end{proof}

We identify  $\alpha_v(1_v)$ with $J(K_v)/\pi J(K_v)$ in the proof of the Lemma below.

\begin{lem}
\label{local}
Suppose that $\chi \in \Xset^p(K_v)$, and $F_v:= \overline{K}_v^{\ker(\chi)}$. Then we have
\begin{equation}
\label{intersection}
\alpha_v(1_v) \cap \alpha_v(\chi) \supseteq (\mathbf{N}(J(F_v)) + \pi J(K_v))/\pi J(K_v),
\end{equation}
where $\N(J(F_v))$ is the image of the norm map $J(F_v) \to J(K_v)$ and the intersection is taken in $H^{1}(K_v, J[\pi])$. In particular, if $\N(J(F_v)) = J(K_v)$, then $h_v(\chi) = 0$.
\end{lem}

\begin{proof}
Consider the commutative diagrams
$$
\begin{tikzpicture}
\matrix(m)[matrix of math nodes,
row sep=2.6em, column sep=1.5em,
text height=1.5ex, text depth=0.25ex]
{J(F_v)/\pi J(F_v)&H^1(F_v, J[\pi])&H^1(F_v, J[\pi])&J(F_v)/\pi J(F_v)\\
J(K_v)/\pi J(K_v)&H^1(K_v, J[\pi])&H^1(K_v, J[\pi])&J(K_v)/\pi J(K_v)\\};
\path[->,font=\scriptsize,>=angle 90]
 (m-2-1) edge node[auto] {$i$} (m-1-1)
 (m-1-1) edge (m-1-2)
(m-2-2)  edge node[right] {res} (m-1-2)
(m-1-3) edge node[auto] {cor} (m-2-3)
(m-1-4) edge node[auto] {$\mathbf{N}$} (m-2-4)
 (m-1-4) edge (m-1-3)
  (m-2-4) edge (m-2-3)
   (m-2-1) edge (m-2-2)
   ;
\end{tikzpicture}
$$

$$
\begin{tikzpicture}
\matrix(m)[matrix of math nodes,
row sep=2.6em, column sep=1.5em,
text height=1.5ex, text depth=0.25ex]
{H^1(F_v, J[\pi])&\times&H^1(F_v, J[\pi])&H^2(F_v, \boldsymbol{\mu}_p)\cong \Z/p\Z\\
H^1(K_v, J[\pi])&\times&H^1(K_v, J[\pi])&H^2(K_v, \boldsymbol{\mu}_p) \cong \Z/p\Z,\\};
\path[->,font=\scriptsize,>=angle 90]

(m-1-1)  edge node[right] {$\mathrm{cor}$} (m-2-1)
(m-2-3) edge node[auto] {$\mathrm{res}$} (m-1-3)
(m-1-3) edge node[auto] {$\cup$}(m-1-4)
 (m-2-3) edge node[auto] {$\cup$}(m-2-4)
(m-1-4) edge node[auto] {$\mathrm{cor} = \mathrm{id}$} (m-2-4)    ;
\end{tikzpicture}\\
$$
where $i$ is the natural map and $\N$ is induced by the norm map.

For convenience, let
\begin{align*}
A &=  \alpha_v(1_v), \\
B &= \alpha_v(\chi), \text{ and }\\
D &= J(F_v)/\pi J(F_v) \cong J^\chi(F_v)/\pi J^\chi(F_v).
\end{align*}

By Lemma \ref{selfdual}, we have $D = D^{\perp} \subseteq \mathrm{res}(A)^{\perp}= \mathrm{cor}^{-1}(A)$, and similarly, $D = D^{\perp} \subseteq  \mathrm{res}(B)^{\perp} = \mathrm{cor}^{-1}(B)$. Therefore, $\N(D) = \mathrm{cor}(D) \subseteq A \cap B.$
\end{proof}

\begin{lem}
\label{aviso}
Let $A$ be an abelian variety defined over $K_v$ such that $\Z[\zeta_p] \subset \End_{G_{K_v}}(A)$. Suppose that $v\nmid p\infty$. Then
$$
A(K_v)/\pi A(K_v) \cong A(K_v)[p^{\infty}]/\pi (A(K_v)[p^{\infty}]).
$$
\end{lem}

\begin{proof}
Multiplication by $\pi$ is surjective on the pro-(prime to $p$) part of $A(K_v)$, so only the pro-$p$ part $A(K_v)[p^\infty]$ contributes to $A(K_v)/\pi A(K_v)$, whence the result follows.
\end{proof}

 \begin{lem}
 \label{ram}
Suppose that $\chi \in \Xset^p(K_v)$, and $F_v = \overline{K}_v^{\ker({\chi})}$, where $v\nmid p\infty$. Then the following hold.
\begin{enumerate}
 \item
$d_p(\alpha_v(1_v))=d_p(J(K_v)/\pi J(K_v))=d_p(J(K_v)[\pi]).$
 \item
 Suppose further that, the extension $F_v/K_v$ is ramified and $J$ has good reduction. Then
$J(K_v)/\pi J(K_v)  \cong  J(F_v)/\pi J(F_v).$
 \end{enumerate}
\end{lem}

 \begin{proof}
We have an exact sequence
$$
\xymatrix@R=5pt@C=15pt{
0 \ar[r] & J(K_v)[\pi] \ar[r] & J(K_v)[p^{\infty}] \ar^-{\pi}[r] & J(K_v)[p^{\infty}] \ar[r] & J(K_v)[p^{\infty}]/\pi  J(K_v)[p^{\infty}] \ar[r] & 0.
}
$$
Hence (i) follows from Lemma \ref{aviso}.

Now we prove (ii). Under the assumptions, $K_v(J[p^{\infty}])$ is an unramified extension over $K_v$. Therefore  $K_v(J(F_v)[p^{\infty}])$ = $K_v$, so $J(K_v)[p^{\infty}] = J(F_v)[p^{\infty}]$. Assertion (ii) follows from passing through (Lemma \ref{aviso})
\begin{align*}
J(K_v)/\pi J(K_v) & \cong J(K_v)[p^{\infty}]/\pi (J(K_v)[p^{\infty}]) \\
& \cong J(F_v)[p^{\infty}]/\pi (J(F_v)[p^{\infty}])  \\
& \cong J(F_v)/\pi J(F_v).
\end{align*}
\end{proof}

\begin{lem}
\label{cycle}
Suppose that $\sigma \in \Gal(f) \subset S_n$ consists of $b$ orbits. Let $J[\pi]^{\sigma =1}$ denote the subgroup of $J[\pi]$ which consists of the elements fixed by $\sigma$. Then $$d_p(J[\pi]^{\sigma =1}) = b-1,$$
\end{lem}

\begin{proof}
Let $\sigma$  be $(\alpha_{11}\alpha_{12}\cdots \alpha_{1i_1})(\alpha_{21}\cdots \alpha_{2i_2})$ $\cdots$ $(\alpha_{b1}\cdots \alpha_{bi_b})$, where $\alpha_{xy}$ are the roots of $f$. We rearrange $\alpha_{xy}$ so that $p\nmid i_b$, which is always possible because $p\nmid n$. Let $a_{xy}$ denote the divisor classes $[(\alpha_{xy}, 0) - \infty]$. Then with the equality $\sum a_{xy} = 0$ (Remark \ref{action}), one can show that
$$
a_{11} + a_{12} +\cdots + a_{1i_1},  a_{21} + \cdots + a_{2i_2}, \cdots\cdots , a_{(b-1)1} + \cdots + a_{b-1i_{b-1}}
$$
form a basis of $J[\pi]^{\sigma =1}$.
\end{proof}

\begin{rem}
Suppose that $v \nmid p\infty$. Let $K_v^{\ur}$ be the maximal unramified extension of $K_v$. It is well-known that if $J/K_v$ has good reduction, then
$J[\pi] \subset J[p] \subset J(K_v^{\ur})$. Let $\Frob_v$ denote the Frobenius automorphism of $K_v^{\ur}$. Restricting $\Frob_v$ to $K(J[\pi])$, one can regard $\Frob_v$ as an element in $S_n$ by Remark \ref{snaction}.
\end{rem}

\begin{lem}
\label{dcyc}
Suppose that $v \nmid p\infty$, and $J$ has a good reduction. Let $b$ be the number of orbits of $\Frob_v \in S_n$. Then
$$
d_p(\alpha_v(1_v)) = b-1.
$$
\end{lem}

\begin{proof}
Note that $J(K_v)[\pi] = J[\pi]^{\Frob_v = 1}$. Then the lemma follows from Lemma \ref{cycle} and Lemma \ref{ram}.
\end{proof}

 \begin{lem}
 \label{localzero}
Let $\chi \in \Xset^p(K_v)$. Suppose that $\chi$ satisfies one of the following conditions:
\begin{itemize}
\item
$\chi$ is trivial, or
\item
$J$ has good reduction, $v \nmid p\infty$, and $\chi$ is unramified.
\end{itemize}
Then $h_v(\chi) = 0$.
 \end{lem}

 \begin{proof}
 Let $F_v = \overline{K}_v^{\ker(\chi)}$. In either case, $\N(J(F_v)) = J(K_v)$ (In second case, it follows from \cite[Corollary 4.4]{norm}). Thus, by Lemma \ref{local}, $h_v(\chi) = 0.$
 \end{proof}

\begin{lem}
\label{ramd}
 Suppose that $\chi \in \Xset^p(K_v)$, and $F_v:=\overline{K}_v^{\ker(\chi)}$, where $v \nmid p\infty$. Suppose that $J(K_v)/\pi J(K_v)  \cong  J(F_v)/\pi J(F_v)$ (which is satisfied if the extension $F_v/K_v$ is ramified and $J$ has good reduction by \text{Lemma }\ref{ram}$)$.
Then
$$
\alpha_v(1_v) \cap \alpha_v(\chi) = \{0 \};
$$
i.e., $h_v(\chi) = \dim_{\Fp}(J(K_v)[\pi])$.

\end{lem}

\begin{proof}
Consider the following commutative diagrams
$$
\xymatrix{
J(F_v)/\pi J(F_v) \ar[r] & H^1(F_v, J[\pi]) & H^1(F_v, J[\pi]) \ar^{\mathrm{cor}}[d] & J(F_v)/\pi J(F_v) \ar^{\N}[d] \ar[l] \\
J(K_v)/\pi J(K_v) \ar^j[r] \ar_{i}^\cong[u] & H^1(K_v, J[\pi]) \ar_{\mathrm{res}}[u]   & H^1(K_v, J[\pi]) & J(K_v)/\pi J(K_v) \ar_j[l]
}
$$
as in the proof of Lemma \ref{local}. The map $i$ in the diagrams is an isomorphism by assumptions.
Let $a \in J(F_v)$. Then $a = \pi b + c$ where $b \in J(F_v)$ and $c \in J(K_v)$. Therefore
$$
\N(a) = \pi\N(b) + pc \in \pi J(K_v).
$$
This means the map $\N$ in the diagrams is actually the zero map. Let
\begin{align*}
A &:= J(K_v)/\pi J(K_v), \\
B &:= J^{\chi}(K_v)/\pi J^{\chi}(K_v), \text{ and }\\
D &:= J(F_v)/\pi J(F_v) \cong J^{\chi}(F_v)/\pi J^{\chi}(F_v),
\end{align*}
for simplicity. From now on, we identify $A, B$ and $D$ with their Kummer images. We need to show that $A \cap B = \{0\}$. In other words, we want to show that $A + B  \text{ ($ =(A\cap B)^{\perp})$} = H^1(K_v, J[\pi])$.
One inclusion is trivial. For the other inclusion, let $x$ be an element in $H^1(K_v, J[\pi])$. According to third diagram in the proof of Lemma \ref{local}, we have $\mathrm{res}^{-1}(D) = (\mathrm{cor}(D))^\perp = 0^\perp = H^1(K_v,J[\pi])$, so we can find $y \in A$ such that $\mathrm{res}(x-y) = 0$. Since $x-y \in \mathrm{ker}(\mathrm{res})$, it is enough to show that
\begin{equation}
\label{NTS}
 \mathrm{ker}(\mathrm{res}) \subseteq A + B.
 \end{equation}
By the Inflation-Restriction Sequence, $$
\mathrm{ker}(\mathrm{res}) \cong H^1(F_v/K_v, J[\pi]^{G_{F_v}}).
$$
But actually, $J[\pi]^{G_{F_v}} = J(F_v)[\pi] = J(K_v)[\pi]$ by Lemma \ref{ram}(i). Hence we have
\begin{equation}
\label{add}
\mathrm{ker}(\mathrm{res}) \cong \Hom(\Gal(F_v/K_v), J(K_v)[\pi]).
\end{equation}
Let $\psi : J \cong J^\chi$ be an isomorphism defined over $F_v$. We have the following commutative diagram
\[
\xymatrix{
& 0\ar[d] \\
& \Hom(\Gal(F_v/K_v), J(K_v)[\pi]) \ar[d] \\
J(K_v) / \pi J(K_v) \ar@{^{(}->}[r]^-{j} \ar[d] & H^1 (K_v, J[\pi]) \ar[d] & J^{\chi} (K_v) /\pi J^{\chi} (K_v) \ar@{_{(}->}[l]_-{j^\chi} \ar[d] \\
J(F_v)/\pi J(F_v) \ar@{^{(}->}[r] \ar @/_2pc/ [rr]^-{\psi}_{\simeq} & H^1(F_v, J[\pi]) & J^\chi (F_v) / \pi J^\chi(F_v), \ar@{_{(}->}[l],
}
\]
where the middle column is exact.
Define a homomorphism
\begin{align*}
\phi: J(K_v)[\pi] & \to \Hom(\Gal(F_v/K_v), J(K_v)[\pi]) \\
P & \mapsto j(\overline{P}) - j^\chi(\overline{\psi(P)}),
\end{align*}
where $\overline{P}$ is the image of $P$ in $J(K_v)/\pi J(K_v)$ and $\overline{\psi(P)}$ is the image of $\psi(P)$ in $J^\chi (K_v) / \pi J^\chi(K_v)$. That the map $\phi$ being well-defined can be shown by diagram chasing in the diagram above. We claim that the map $\phi$ is an isomorphism, which will imply \eqref{NTS}. It is enough to show the map $\phi$ is injective since $J(K_v)[\pi]$ and $\Hom(F_v/K_v, J(K_v)[\pi])$ have the same dimension over $\Fp$. It can be easily checked that for $P \neq 0$,
$$
(j(\overline{P}) - j^\chi(\overline{\psi(P)}))(\tau) \neq 0,
$$
where $\tau$ is a nontrivial element in $\Gal(F_v/K_v)$, so we are done.
\end{proof}

Until the end of this section, assume that $p=2$. We show that the equality holds in \eqref{intersection} when $p=2$. For the elliptic curve case, the following proposition is proved in \cite[Proposition 7]{kramer} and \cite[Proposition 5.2]{MR2}.

\begin{prop}
\label{localintersection}
Let $\chi
 \in \Xset^2(K_v)$. Then
\begin{equation}
\alpha_v(1_v) \cap \alpha_v(\chi) = \N J(L_v)/2J(K_v),
\end{equation}
where $L_v = \overline{K}_v^{\ker(\chi)},$ and $\N J(L_v)$ is the image of the norm map $\N: J(L_v) \to J(K_v)$.
\end{prop}
The following two lemmas are used to prove the proposition.

\begin{lem}
\label{dim1}
Suppose that $\chi \in \Xset^2(K_v)$ is a nontrivial quadratic character, and $L_v:= \overline{K}_v^{\ker(\chi)}$. Let
$$
\phi: H^1(K_v, J[2]) \to H^1(L_v, J[2])
$$
be the restriction map. Then $(i) \text{ } \ker(\phi)\subseteq \alpha_v(1_v) + \alpha_v(\chi),$ and $(ii) \text{ } d_2(\ker(\phi)) = 2d_2(J(K_v)[2]) - d_2(J(L_v)[2])$.
\end{lem}

\begin{proof}
Let
\begin{align*}
i:& J(K_v)/2J(K_v) \to H^1(K_v, J[2]), \text{ and } \\
i^\chi:& J^{\chi}(K_v)/2J^{\chi}(K_v) \to H^1(K_v, J^{\chi}[2]) \cong H^1(K_v, J[2]).
\end{align*}
By the Inflation-Restriction Sequence,
$$
\ker(\phi) =  H^1(L_v/K_v, J(L_v)[2])
=  J(K_v)[2]/(\tau -1)J(L_v)[2],
$$
where $\Gal(L_v/K_v)$ is generated by $\tau$. Then for any $P \in J(K_v)[2]$, its image $c_P$ in $H^1(K_v, J[2])$ of the composition map
$$
J(K_v)[2] \to J(K_v)[2]/(\tau -1)J(L_v)[2] \subset H^1(K_v, J[2])
$$
is given by $c_P(\sigma) = P$ if $\sigma|_{L_v} = \tau$, and $c_P(\sigma) =0$ otherwise. With the isomorphism $J(\overline{K}_v) \cong J^\chi(\overline{K}_v)$, it is straightforward to check that $i(\overline{P}) + i^{\chi}(\overline{P^\chi}) = c_P$, where $\overline{P}$ and $\overline{P^\chi}$ represent the images of $P$ in $J(K_v)/2J(K_v)$ and $J^\chi(K_v)/2J^\chi(K_v)$, respectively. This proves (i). The exact sequence
$$
\xymatrix{
0 \ar[r] & J(K_v)[2] \ar[r] & J(L_v)[2] \ar^{\tau -1}[r] & J(K_v)[2] \ar[r] & \ker(\phi) \ar[r] &0
}
$$
shows (ii).
\end{proof}

\begin{lem}
\label{dim2}
Let $L_v/K_v$ be a nontrivial quadratic extension. Then
$$
2d_2(J(K_v)[2]) - d_2(J(L_v)[2])= 2d_2(J(K_v)/2J(K_v)) - d_2(J(L_v)/2J(L_v)).
$$
\end{lem}

\begin{proof}
Consider the following diagram
$$
\xymatrix{
0 \ar[r] & J^\chi(K_v) \ar[r] \ar[d]^2 & J(L_v) \ar^N[r] \ar[d]^2 & J(K_v)  \ar[r] \ar^{2}[d]& J(K_v)/N(J(L_v)) \ar[r] \ar[d]^2 & 0 \\
0 \ar[r] & J^\chi(K_v) \ar[r] & J(L_v) \ar^N[r] & J(K_v) \ar[r] & J(K_v)/N(J(L_v)) \ar[r] & 0.
}
$$
Note that each row consists of two short exact sequences
$$
\xymatrix@R=5pt@C=15pt{
0 \ar[r] & J^\chi(K_v) \ar[r] & J(L_v) \ar^N[r]
    & NJ(L_v) \ar[r] & 0, \text{ and }  \\
0 \ar[r] & NJ(L_v) \ar[r] & J(K_v) \ar[r] & J(K_v)/N(J(L_v)) \ar[r] & 0.
}
$$
Then applying the snake lemma for each short exact sequence and comparing the dimensions over $\Ftwo$ together with Proposition \ref{canonicaliso} and Remark \ref{twistselfdual} show the result.
\end{proof}

\begin{proof}[Proof of Proposition $2.17$]
If $L_v/K_v$ is trivial, then it is obvious. Thus assume $L_v/K_v$ is a nontrivial quadratic extension.
By Lemma \ref{local}, it is enough to show that
\begin{equation}
\label{goodgood}
d_2(\alpha_v(1_v) \cap \alpha_v(\chi)) = d_2(\N J(L_v)/2J(K_v)).
\end{equation}
Consider the following exact sequence $(M:=J(K_v) + J^\chi(K_v) + 2J(L_v))$
\[
\xymatrixrowsep{1in}
\xymatrix@R=5pt@C=15pt{ 0 \ar[r] & M/2J(L_v) \ar[r] & J(L_v)/2J(L_v) \ar^{\N}[r] & J(K_v)/2J(K_v) \ar[r] & J(K_v)/N(J(L_v)) \ar[r] & 0,
}
\]
where the middle map $\N$ is induced by the norm map and $J^\chi(K_v)$ is regarded as a subgroup of $J(L_v)$. For simplicity, write $X, Y, Z$ and $W$ for the nontrivial terms in the exact sequence in order. Let $A = \alpha_v(1_v),$ and $B = \alpha_v(\chi)$. Then $d_2(A + B) + d_2(A \cap B) = d_2(H^1(K_v, J[2]))$ since $A\cap B$ is the orthogonal complement of $A+B$ in \eqref{tld}.
Let
$$
\phi : H^1(K_v, J[2]) \to H^1(L_v, J[2])
$$
be the restriction map. We have $X = \phi(A+B)$, so by Remark \ref{twistselfdual}, Lemma \ref{dim1}, and Lemma \ref{dim2},
\begin{align*}
d_2(X) &= d_2(A + B) - d_2(\ker(\phi))\\
& =  d_2(H^1(K_v, J[2])) - d_2(A \cap B) - d_2(\ker(\phi)) \\
&= d_2(H^1(K_v, J[2])) - d_2(A \cap B) - 2d_2(J(K_v)[2]) + d_2(J(L_v)[2]) \\
&= d_2(J(L_v)/2J(L_v)) - d_2(A \cap B).
\end{align*}
Then the equality $d_2(X) + d_2(Z) = d_2(Y) + d_2(W)$ shows \eqref{goodgood}, as desired.
\end{proof}

For the following lemma, we specify $v_0$ so that $K_{v_0} = \R$.

\begin{lem}
\label{reallocal}
Let $C_{2,f}$ be a hyperelliptic curve over the real field $\R$. Let $\eta$ be the sign character. Suppose that $f$ has $2k_1 -1$ real roots and $2k_2$ complex roots. Then
\begin{enumerate}
\item
$J(\R) \cong (\R/\Z)^{k_1 + k_2 -1} \oplus (\Z/2\Z)^{k_1-1}, \text{ and }$
\item
$h_{v_0}(\eta) = k_1 - 1$
\end{enumerate}
\end{lem}

\begin{proof}
Complex conjugation corresponds to the element of $S_n$ that consists of $2k_1-1$ cycles of length $1$ and $k_2$ cycles of length $2$. Hence \cite[Remark I.3.7]{milne} and Lemma \ref{cycle} show (i). Note that the image of the norm map $\N: J(\C) \to J(\R)$ is the the connected component of $0$ ($=(\R/\Z)^{k_1 + k_2 -1}$) according to \cite[Remark I.3.7]{milne}. Then (ii) follows from Proposition \ref{localintersection}.
\end{proof}



\section{Selmer groups and controlling the localization map}
Let $C_{p,f}$ be a superelliptic curve over a number field $K$ containing $\zeta_p$, where $n = \deg(f)$ is not divisible by $p$. Let $v$ be a place of $K$. We write $K_v$ for the completion of $K$ at $v$. For the rest of the paper, for every prime $\l$, we fix an embedding $\overline{K} \hookrightarrow \overline{K}_{\l}$, so that $G_{K_{\l}} \subset G_K$. As usual, we write $J$ for the Jacobian of $C_{p,f}$. Recall that $\Gal(f) \in S_n$ has a natural action on $J[\pi]$ (Remark \ref{snaction}).

\begin{defn}
Let $\chi \in \Xset^p(K)$. The $\pi$-Selmer group $\Sel_\pi(J^\chi/K) \subset H^1(K,J[\pi])$ is the (finite)
$\F_p$-vector space defined by the following exact sequence
$$
0 \too \Sel_\pi(J^\chi/K) \too H^1(K,J[\pi]) \too \dirsum{v}H^1(K_v,J[\pi])/\alpha_v(\chi_v),
$$
where the rightmost map is the direct sum of the restriction maps, and $\chi_v$ is the restriction of $\chi$ to $G_{K_v}$.
\end{defn}

\begin{lem}
\label{cohcom}
Suppose that $p \nmid n$. Then $H^1(S_n, J[\pi]) = 0$.
\end{lem}

\begin{proof}
We regard $S_n$ as the symmetric group on the set $\{1,2,\cdots, n\}$. Let $H:=\{\sigma \in S_n\mid \sigma(n) = n\} \cong S_{n-1} \subset S_n$. We have an $S_n$-module isomorphism
$$
\mathrm{Ind}_{H}^{S_{n}}\Fp \cong \Fp[S_n/H]
$$
where $\mathrm{Ind}_{H}^{S_{n}}\Fp$ denote the induced $S_n$-module of the $H$-module $\Fp$ (trivial action of $H$ on $\Fp$).
Let $\beta_i$ be ($i$ $n$) in $\Fp[S_n/H]$. Note that $\sigma(\beta_i) = \beta_{\sigma(i)}$ in $S_n/H$. We put
$$
D =\{a(\beta_1 + \beta_2 + \cdots + \beta_n) \mid a \in \Fp \}.
$$
Then
$$
\mathrm{Ind}_{H}^{S_{n}}\Fp /D \cong \Fp[S_n/H]/D \cong J[\pi].
$$
In the last isomorphism the map is defined by $\beta_i \mapsto \alpha_i$ where $\alpha_i$ are the roots of polynomial of the given superelliptic curve. We have an ($S_n$-module) exact sequence
\begin{equation}
\label{exactseq}
\xymatrix@R=5pt@C=15pt{
0 \ar[r] & D \ar^-{j}[r] & \mathrm{Ind}_{H}^{S_{n}}\Fp \ar[r] & \mathrm{Ind}_{H}^{S_{n}}\Fp/D \ar[r]  & 0
}
\end{equation}
But since $p\nmid n$, we have a map
$$
g : \mathrm{Ind}_{H}^{S_{n}}\Fp \too D
$$
defined by $a_1\beta_1 + \cdots + a_n\beta_n \mapsto n^{-1}(a_1+\cdots + a_n)(\beta_1 + \cdots \beta_n)$ where $n^{-1}$ is taken in $(\Z/p\Z)^\times$. Clearly $ g\circ j =\text{id}_D$, so the exact sequence \eqref{exactseq} splits. Hence
$$
H^1(S_n, \mathrm{Ind}_H^{S_n}\Fp) \cong H^1(S_n, D) \oplus H^1(S_n, J[\pi]).
$$
By Shapiro's lemma, $H^1(S_n, \mathrm{Ind}_H^{S_n}\Fp) \cong H^1(H, \Fp)$. If $p$ is odd, $H^1(H, \Fp) = \Hom(H, \Fp) = 0$. If $p=2$ (so $n \ge 3$), $H^1(H, \Ftwo) = \Z/2\Z = \Hom(S_n, D) = H^1(S_n, D)$. In either case, we have $H^1(S_n, J[\pi]) = 0$.
\end{proof}

\begin{defn}
\label{star}
We say that $C_{p,f}$ satisfies ($\ast$) if one of the following holds.
\begin{itemize}
\item
$p=2$, $\Gal(f) \cong A_n$ or $S_n$, and $n \ge 5$.
\item
$p=2, n=3$, and $\Gal(f) \cong S_3$.
\item
$p$ is an odd prime, and $\Gal(f) \cong S_n$.
\end{itemize}
\end{defn}

If $c \in H^1(K, J[\pi])$ and $\sigma \in G_K$, let
$$
c(\sigma) \in J[\pi]/(\sigma - 1)J[\pi]
$$
denote the image of $\sigma$ under any cocycle representing $c$.

\begin{lem}
\label{3.5}
Let $N$ be a finite subgroup of $H^1(K, J[\pi])$ and $\sigma \in G_K$. Suppose that $\phi : N \too J[\pi]/(\sigma - 1)J[\pi]$ is a homomorphism.
Suppose that $\Gal(K(J[\pi])/K) = \Gal(f) = \Omega$ satisfies the following conditions.
\begin{enumerate}
 \item
 $H^{1}(\Omega, J[\pi])$ = $0$,
 \item
 $J[\pi]$ is a simple $\Omega$-module,
 \item
 $\dim_{\Fp}(\Hom_{\Omega}(J[\pi], J[\pi])) = 1$, \text{ and }
 \item
 $\Omega$ does not act on $J[\pi]$ trivially.
  \end{enumerate}
Then there exists an element $\rho \in G_K$ such that $\rho \mid _{K(J[\pi])K^{ab}} = \sigma \mid _{K(J[\pi])K^{ab}}$ and $c(\rho) = \phi(c)$ for all $c \in N$.  In particular, if $C_{p,f}$ satisfies $(\ast)$, then there exists such an element $\rho \in G_K$.
\end{lem}

\begin{proof}
The proof of Lemma 3.5 of \cite{hilbert} works here, too. For the last assertion, it is not difficult to check that (ii), (iii), and (iv) are satisfied when $C_{p,f}$ satisfies $(\ast)$. If $\Gal(f) \cong S_n$, the condition (i) is Lemma \ref{cohcom}. If $p=2$, and $\Gal(f) \cong A_n$, then $J[2]^{A_n = 1} = 0$. The Hochschild-Serre Spectral Sequence (for example, see \cite[Proposition 1.6.7]{cohomology}) together with the fact that $H^1(S_n, J[2]) =0$ shows that
$H^1(A_n, J[2])^{S_n/A_n = 1} = 0$. Then $H^1(A_n, J[2]) = 0$ by the following fact that can be proved by a standard argument of group theory:
If $U, V$ are nontrivial $2$-groups such that $U$ acts on $V$, then $V^U$ (:=the group of elements in $V$ fixed by every element in $U$) is non-trivial.
\end{proof}

\begin{rem}
Note that if $C_f$ is an elliptic curve, and $\Gal(f) \cong A_3$, the condition (iii) in Lemma \ref{3.5} does not hold.
\end{rem}

\begin{defn}
\label{locmap}
For every place $v$ of $K$, we write $\mathrm{res}_v$ for the restriction map (fixing an embedding $\overline{K} \hookrightarrow \overline{K}_v$)
$$
\mathrm{res}_v: H^1(K, J[\pi]) \to H^1(K_v, J[\pi]).
$$
Suppose that $J$ has good reduction at $\l\nmid p\infty$. Define the localization map
$$
\mathrm{loc}_\l : \Sel_\pi(J/K) \to \alpha_\l(1_\l) \cong J[\pi]/(\Frob_\l -1),
$$
where the former map is $\mathrm{res}_\l$ and the latter map given in Lemma \ref{hardtofindref}. Note that $\mathrm{loc}_\l$ is given by evaluating cocycles at a Frobenius automorphism $\Frob_\l$.
\end{defn}

We define various Selmer groups as follows.

\begin{defn}
Let $\l$ be a prime of $K$ and $\psi_\l \in \Xset^p(K_\l)$. Define
\begin{align*}
\Sel_{\pi}(J, \psi_\l) := \{x \in H^1(K, J[\pi])|& \mathrm{res}_v(x) \in \alpha_v(1_v) \text{ if }v \neq \l, \text{ and }\\
 &\mathrm{res}_\l(x) \in \alpha_\l(\psi_\l)  \}.
\end{align*}
Define
\begin{align*}
\Sel_{\pi, \l}(J/K) := \{x \in H^1(K, J[\pi])|& \mathrm{res}_v(x) \in \alpha_v(1_v) \text{ if }v \neq \l, \text{ and }\\
 &\mathrm{res}_\l(x) = 0   \}.
\end{align*}
Define
$$
\Sel_{\pi}^{\l}(J/K) := \ker(H^1(K, J[\pi]) \to \bigoplus_{v \neq \q} H^1(K_v, J[\pi])/\alpha_v(1_v)).
$$
\end{defn}
Obviously, $ \Sel_{\pi, \l}(J/K) \subseteq \Sel_{\pi}(J/K) \subseteq \Sel_{\pi}^{\l}(J/K)$.

\begin{lem}
\label{ptd}
The right hand restriction maps of the following exact sequences are orthogonal complements under \eqref{tld}.
$$
\xymatrix@R=5pt@C=15pt{
0 \ar[r] & \Sel_\pi(J/K) \ar[r] & \Sel_\pi^\l(J/K) \ar[r]
    & H^1(K_v,J[\pi])/\alpha_\l(1_\l),  \\
0 \ar[r] & \Sel_{\pi, \l}(J/K) \ar[r] & \Sel_\pi(J/K) \ar[r] & \alpha_\l(1_\l).
}
$$
In particular, $d_p(\Sel_\pi^\l(J/K)) - d_p(\Sel_{\pi, \l}(J/K)) = d_p(\alpha_\l(1_\l)) = \frac{1}{2}d_p(H^1(K_\l, J[\pi])).$
\end{lem}

\begin{proof}
The lemma follows from the Global Poitou-Tate Duality. For example, see \cite[Theorem 1.7.3]{EulerSystems}.
\end{proof}

\section{$\pi$-Selmer ranks of Jacobians of superelliptic curves}
 We continue to assume that $C_{p,f}$ is a superelliptic curve over a number field $K$ containing $\zeta_p$. In this section, we assume that $p$ is a prime and $\Gal(f) \cong S_n,$ where $n = \deg(f)$ is not divisible by $p$. Let $\Sigma$ denote a (finite) set of places which contains all places where $J$ has bad reduction, all archimedean places, and all primes above $p$. We enlarge $\Sigma$, if necessary, so that $\Pic(\O_{K, \Sigma}) = 0$. As before, we write $J$ for the Jacobian of $C_{p,f}$.  In this section, we prove that if $p$ is an odd prime, there exist infinitely many $p$-twists of $J$ whose Jacobians have $\pi$-Selmer ranks equal to any non-negative integer $r$.

\begin{rem}
Note that our $p$-twists (Definition \ref{definitionofquadratictwist}) of $J$ over $K$ are also Jacobians of superelliptic curves over $K$ and that a $p$-twist of a $p$-twist is still a $p$-twist. This enables us to use an inductive argument. For example, if we have an algorithm to construct a $p$-twist of the Jacobian of a superelliptic curve having a strictly bigger $\pi$-Selmer group than the original $\pi$-Selmer group, we can make the $\pi$-Selmer group as big (the dimension over $\Fp$) as we want by taking $p$-twists.
\end{rem}

\begin{prop}
\label{superelliptic}
Suppose that $K$ is a number field containing $\zeta_p$, and $f \in K[x]$ is a separable polynomial. Let $n = \deg(f)$ and suppose that $p\nmid n$ is an odd prime and $\Gal(f) \cong S_n$. Let $J:= J(C_{p,f})$. Suppose that $d_p(\Sel_\pi(J/K)) \ge 1$. Then there exist infinitely many $p$-twists $J^\chi$ such that $d_p(\Sel_\pi(J^\chi/K)) = d_p(\Sel_\pi(J/K)) - 1$.
\end{prop}

\begin{proof}
We prove this proposition by following the method of the proof of \cite[Proposition 5.1]{hilbert}. Let $\Delta_f$ be the discriminant of the polynomial $f$. Let $\theta$ be the (formal) product of $p^3$, all primes where $J$ has bad reduction and all archimedean places. Let $K(\theta)$ be its ray class field and $K[\theta]$ be the $p$-maximal subextension of $K(\theta)$. Then, $K[\theta]$ and $K(J[\pi])$ are linearly disjoint. Indeed, $S_n$ has no normal subgroup of index $p$ for an odd prime $p$. Therefore we can find an element $\sigma \in G_K$ such that $\sigma \mid _{J(K[\pi])}$ consists of $2$ disjoint orbits $(\sigma|_{J(K[\pi])} \in \Gal(K(J[\pi])/K) = \Gal(f) = S_n)$, and $\sigma \mid _{K[\theta]} = 1$.

  Let
  $$
  \phi : \Sel_{\pi}(J/K) \too J[\pi]/(\sigma-1)J[\pi]
  $$
  be a homomorphism. By Lemma \ref{3.5}, we can find $\rho \in G_K$ such that
  \begin{itemize}
  \item
  $\rho \mid _{K[\theta]K(J[\pi])} = \sigma$
  \item
  $c(\rho) = \phi(c)$ for every $c \in \Sel_{\pi}(J/K)$.
  \end{itemize}

  Let $N$ be a Galois extension of $K$ containing $K(\theta)K(J[\pi])$, large enough so that the restriction of every element in $\Sel_{\pi}(J/K)$ to $G_N$ is zero (Choosing such a $G_N$ is possible because the Selmer group is finite).

  By the Chebotarev density theorem, we can find a prime $\l$ of $K$ such that $\l\nmid p$, $J$ has good reduction at $\l$, the extension $N/K$ is unramified at $\l$ and $\Frob_\l \in$ (the conjugacy class of $\rho$ in $\Gal(N/K)$). Since $p \nmid [K(\theta) : K[\theta]]$, the restriction of $\rho^k$  to ${K(\theta)}$ is trivial for some $p\nmid k$. Therefore $\l^k$ is principal generated by $d \equiv 1$(mod $\theta$). Let $F = K(\sqrt[p]{d})$. Then all places dividing $\theta$ split in $F/K$, the extension $F/K$ is ramified at $\l$, and nowhere else because its discriminant divides $p^pd^{p-1}$. Let $\chi$ denote the image of the Kummer map $K^\times/(K^\times)^p \cong \Xset^p(K)$. Therefore, by Lemma \ref{localzero}, $\alpha(1_v) = \alpha(\chi_v)$ for all places $v$ except $\l$, where $\chi_v$ is the restriction of $\chi$ to $G_{K_v}$.

  Since $J$ has good reduction at $\l$,
  $$
  \alpha_\l(1_\l) \cong J[\pi]/(\Frob_{\l}-1)J[\pi] \cong J[\pi]/(\rho-1)J[\pi] = J[\pi]/(\sigma-1)J[\pi].
  $$
  The first isomorphism follows from Lemma \ref{hardtofindref}.
  By Lemma \ref{dcyc} and our choice of $\sigma$, we have $d_p(\alpha_\l(1_\l)) = 1$.

The localization map (Definition \ref{locmap})
  $$
  \loc_\l : \Sel_{\pi}(J/K) \too \alpha_\l(1_\l) \cong J[\pi]/(\rho-1)J[\pi]
  $$
  is given by evaluation of cocycles at $\Frob_\l \sim \rho$. Therefore we have $$
  \loc_\l(\Sel_{\pi}(J/K)) = \phi(\Sel_{\pi}(J/K))
  $$
by Lemma \ref{3.5}.
Note that $d_p(\Sel_\pi^\l(J/K))-d_p(\Sel_{\pi, \l}(J/K)) = 1$ by Lemma \ref{ptd}.
Choose $\phi$ so that $d_p(\mathrm{Im}(\phi)) = 1$. Then,
$\Sel_\pi(J/K) = \Sel_\pi^\l(J/K).$
Then Lemma \ref{ramd} and Lemma \ref{ptd} show that
$$
d_p(\Sel_\pi(J^\chi/K)) = d_p(\Sel_\pi(J/K)) - 1.
$$
\end{proof}

\begin{defn}
\label{cpi}
Let $\theta$ be a (formal) product of primes of $K$. Define
\begin{align*}
\Sigma(\theta) & := \Sigma \cup \{\l : \l|\theta \}, \\
\cP_{i} &:=\; \{\l : \text{$\l \notin \Sigma$
   and $\dim_{\Fp}(J[\pi]^{G_{K_\l}}) = i$}\}\quad \text{for $0 \le i \le n-1$}, \text{ and} \\
\cP &:=\; \cP_0 \textstyle\coprod \cP_1 \coprod \cP_2 \textstyle\coprod \cdots \textstyle\coprod \cP_{n-1} = \{\l : \l \notin \Sigma \}.
\end{align*}
\end{defn}

\begin{rem}
\label{localdim}
If $\l \in \cP_i$, then $d_p(\alpha_\l(1_\l)) = i$ by Lemma \ref{ram}.
\end{rem}

\begin{lem}
\label{pinftorsion}
Suppose that $v$ is a prime of $K$ such that $v\nmid p\infty$, and $J$ has good reduction at $v$. If $\psi_v \in \Xset^p_\ram(K_v)$, then
$$
J^{\psi_v}(K_v)[p^\infty] = J^{\psi_v}(K_v)[\pi] \text{ }(\cong J(K_v)[\pi])
$$
\end{lem}

\begin{proof}
Let $L_v := \overline{K}_v^{\mathrm{Ker}(\psi_v)}$ so that $L_v$ is a (totally) ramified extension over $K_v$ of degree $p$. Let $K_v^\ur, L_v^\ur$ denote the maximal unramified extensions over $K_v, L_v$, respectively.
It is sufficient to prove that
\begin{equation}
\label{pipipi}
J^{\psi_v}(K_v^\ur)[p^\infty] = J^{\psi_v}(K_v^\ur)[\pi].
\end{equation}
Let $\sigma \in \Gal(L_v^\ur/K_v^\ur)$ be non-trivial. Then
$$
J^{\psi_v}(L_v^\ur)[p^\infty]^{\sigma=1} = J^{\psi_v}(K_v^\ur)[p^\infty].
$$
It is well-known that the assumptions that $v\nmid p\infty$ and $J$ has good reduction at $v$ imply that $J[p^\infty] \subset J(K_v^\ur)$. Let
$$
\lambda : J^{\psi_v} \cong J
$$
be an isomorphism over $L_v$. Note that $\lambda^\sigma = \sigma\lambda\sigma^{-1} = \psi_v(\sigma)\lambda$. If $P \in J^{\psi_v}(K_v^{\ur})[p^\infty]$, then
\begin{equation}
\label{pinfty}
\lambda(P) = \lambda(P^{\sigma}) = \psi_v(\sigma)^{-1}\lambda^\sigma(P^{\sigma}) =  \psi_v(\sigma)^{-1}(\lambda(P))^{\sigma} = \psi_v(\sigma)^{-1}\lambda(P),
\end{equation}
whence $P = \zeta_p P$ if we take $\sigma$ so that $\psi_v(\sigma) = \zeta_p^{-1}$. In the last equality in \eqref{pinfty}, we have used the fact that $J[p^\infty] \subset J(K_v^\ur)$. Now it is easy to see that \eqref{pipipi} holds.
\end{proof}

\begin{defn}
\label{characterdefine}
Choose a nontrivial unramified character $\epsilon_\l \in \Xset^p(K_\l)$ and a ramified character $\eta_\l \in \Xset_{\ram}(K_\l)$ for every prime $\l \in \cP$. Define
$$
\eta_{\l, j} := \eta_\l\epsilon_\l^j
$$
for every $0 \le j  \le p-1$.
\end{defn}
Obviously, all $\eta_{\l, j}$ are in $\Xset^p_{\ram}(K_\l)$.
\begin{lem}
\label{rlcd}
 Suppose that $p$ is an odd prime. Let $\l \in \cP$ be a prime such that every orbit of $\Frob_\l \in S_n$  has length not divisible by $p$. Then for any $a, b$ such that $0 \leqq a,b \leqq p-1$ and $a \neq b$,
$$
\alpha_\l(\eta_{\l, a}) \cap \alpha_\l(\eta_{\l, b}) = \{0\}.
$$
\end{lem}

\begin{proof}
Lemma \ref{aviso} and Lemma \ref{pinftorsion} show that
\begin{align*}
J^{\eta_{\l,a}}(K_\l)/\pi J^{\eta_{\l,a}}(K_\l) &\cong J^{\eta_{\l,a}}(K_\l)[p^\infty]/\pi J^{\eta_{\l,a}}(K_\l)[p^\infty] \\
 & \cong J^{\eta_{\l,a}}(K_\l)[\pi]/\pi J^{\eta_{\l,a}}(K_\l)[\pi] \\
 & = J^{\eta_{\l,a}}(K_\l)[\pi].
\end{align*}
Let $F_\q := \overline{K}_\l^{\text{Ker } \epsilon_\l^{b-a}}$, the degree $p$ unramified extension over $K_\l$. Since every orbit of $\Frob_\l$ has length not divisible by $p$, the degree $[K_\l(J[\pi]) : K_\l]$ is not as well. In particular, $F_\l$ and $K_\l(J[\pi])$ are linearly disjoint over $K_\l$. Then it follows that
$$
J^{\eta_{\l,a}}(F_\l)[\pi] \cong J(F_\l)[\pi] = J(K_\l)[\pi] \cong J^{\eta_{\l,a}}(K_\l)[\pi].
$$
Hence we have
$$
J^{\eta_{\l,a}}(F_\l)/\pi J^{\eta_{\l,a}}(F_\l) \cong J^{\eta_{\l,a}}(F_\l)[\pi] = J^{\eta_{\l,a}}(K_\l)[\pi] \cong J^{\eta_{\l,a}}(K_\l)/\pi J^{\eta_{\l,a}}(K_\l),
$$
where the first isomorphism comes exactly as above. Finally, we apply Lemma \ref{ramd} to get the conclusion.
\end{proof}

\begin{lem}\cite[Lemma 6.6]{KMR}
\label{KMRlemma}
Suppose that $G, H$ are abelian groups and $M \subset G \times H$ is a subgroup. Let $\pi_G$ and $\pi_H$ denote the projection maps from $G \times H$ to $G$ and $H$, respectively. Let $M_0:= \ker(\pi_G:M \to G/G^p)$.
\begin{enumerate}
\item
The image of the natural map $\Hom((G\times H)/M, \boldsymbol{\mu}_p) \to \Hom(H,\boldsymbol{\mu}_p))$ is exctly $\{g \in \Hom(H,\boldsymbol{\mu}_p): \pi_H(M_0) \subset \ker(g)\}$
\item
If $M/M^p \to G/G^p$ is injective, then $\Hom((G\times H)/M, \boldsymbol{\mu}_p) \to \Hom(H, \boldsymbol{\mu}_p)$ is surjective.
\end{enumerate}
\end{lem}

\begin{lem}
\label{pzeroinj}
Suppose that $\O_\l$ is the ring of integers of $K_\l$ for every prime ideal $\l$. Then the natrual map
\begin{equation}
\label{injmap}
\O_{K,\Sigma}^\times/(\O_{K,\Sigma}^\times)^p \too \prod_{\l \in \cP_0} \O_\l^\times/(\O_{\l}^\times)^p
\end{equation}
is injective.
\end{lem}

\begin{proof}
Let $\alpha$ be a nontrivial element of $\O_{K,\Sigma}^\times/(\O_{K,\Sigma}^\times)^p$. We want to find a prime $\l' \in \cP_0$ such that $\alpha \notin (\O_{\l'}^\times)^p$. Two fields $K(\sqrt[p]{\alpha})$ and $K(J[\pi])$ are linearly disjoint over $K$ because $S_n$ doesn't have a normal subgroup of index $p$ for an odd prime $p$. We choose an element $\tau \in G_K$ such that
\begin{itemize}
  \item
  $\tau \mid _{K(\sqrt[p]{\alpha})} \neq 1, \text{ and }$
  \item
  $\tau \mid _{K(J[\pi])} \in S_n \text{ is a $n$-cycle}.$
  \end{itemize}
Let $U$ be the conjugacy class of $\tau$ in $\Gal(K(\sqrt[p]{\alpha})K(J[\pi])/K)$. By the Chebotarev density theorem, there exist infinitely many $\l'$ such that $\Frob_{\l'} \in U$. Such a prime $\l'$ satisfies both $\alpha \notin (\O_{\l'}^\times)^p$ and $\l' \in \cP_0$ if $\l' \not\in \Sigma$.
\end{proof}

\begin{rem}
\label{largerthansigma}
One can check easily that Lemma \ref{pzeroinj} still holds if we replace $\Sigma$ by a finite set containing $\Sigma$.
\end{rem}

\begin{prop}
\label{amazingprime}
Suppose that $p$ is an odd prime. Then there are infinitely many prime ideals $\ell \in \cP_1$ and an integer $0\le e \le p-1$ so that
\begin{equation}
\label{incselrk}
d_p(\Sel_{\pi}(J, \eta_{\ell, e})) = d_p(\Sel_{\pi}(J/K)) + 1,
\end{equation}
where $\eta_{\ell, e}$ is as in Definition \ref{characterdefine}.
\end{prop}

\begin{proof}
Since $\Gal(K(J[\pi]/K) \cong S_n$, we can find a prime $\ell \in \cP_1$ (equivalently, $\Frob_\ell$ has $2$ orbits) such that neither of the order of the orbits of $\Frob_\ell\in S_n$ is divisible by $p$ (which is possible since $p$ is an odd prime) and that the localization map (Definition \ref{locmap})
$$
\loc_\ell : \Sel_{\pi}(J/K) \too \alpha_\ell(1_\ell) \cong J[\pi]/(\Frob_\ell-1)J[\pi]
$$
is trivial by Lemma \ref{3.5} combined with the Cheboterev Density Theorem. In other words, $\Sel_{\pi}(J/K) = \Sel_{\pi, \ell}(J/K)$. By Lemma \ref{ptd}, $\dim_{\Fp}(\Sel_{\pi}^\ell(J/K)) = \dim_{\Fp}(\Sel_{\pi, \ell}(J/K)) + 1$. Denote the image of the restriction map
$$
\mathrm{res}_{\ell}: \Sel_{\pi}^{\ell}(J/K)) \to H^1(K_\ell, J[\pi])
$$
by $\mathrm{res}_{\ell}(\Sel_{\pi}^{\ell}(J/K))$. Then the set $\mathrm{res}_{\ell}(\Sel_{\pi}^{\ell}(J/K))$ is a $1$-dimensional $\Fp$-vector subspace of $H^1(K_{\ell}, J[\pi])$. But Lemma \ref{rlcd}, Lemma \ref{ramd}, and Lemma \ref{localzero} together show that one can find a $\eta_{\ell, e} \in \Xset^p_{\ram}(K_\ell)$ such that
$$\mathrm{res}_\ell(\Sel_{\pi}^\ell(J/K)) = \alpha_\ell(\eta_{\ell, e}),$$
since there are exactly $p+1$ pairwise distinct $1$-dimensional subspaces of $H^1(K_\ell, J[\pi])$ because $\ell \in \cP_1$ (so $d_2(H^1(K_\ell, J[\pi]) = 2$). It follows that $\Sel_{\pi}^\ell(J/K) = \Sel_{\pi}(J, \eta_{\ell, e})$, so we are done.
\end{proof}

\begin{prop}
\label{hoohoo}
Suppose that $K$ is a number field containing $\zeta_p$, and $f \in K[x]$ is a separable polynomial. Let $n = \deg(f)$ and suppose that $p\nmid n$ is an odd prime and $\Gal(f) \cong S_n$. Let $J:= J(C_{p,f})$. Then there exists infinitely many $p$-twists $J^\chi$ such that $d_p(\Sel_\pi(J^\chi/K)) = d_p(\Sel_\pi(J/K)) + 1$.
\end{prop}

\begin{proof}
The main technique in the proof is already used in that of \cite[Proposition 6.8]{KMR}. Suppose that $\ell \in \cP_1$ is as in Proposition \ref{amazingprime}.
As stated earlier in this section we can enlarge $\Sigma$, if necessary, so that $\Pic(\O_{K,\Sigma(\ell)}) = 0$.
Thus global class field theory gives
$$
\Xset^p(K) = \Hom(\iK/K^\times,\boldsymbol{\mu}_p) = \textstyle\Hom((\prod_{v \in \Sigma(\ell)}K_v^\times \times
      \prod_{\l\notin\Sigma(\ell)}\O_\l^\times)/\O_{K,\Sigma(\ell)}^\times,\boldsymbol{\mu}_p).
$$
Let $\psi_{\ell, e} \in \Xset^p_{\ram}(K_\ell)$ be the character such that \eqref{incselrk} holds. Define
\begin{align*}
Q &:= \cP - \{\cP_0 \cup \Sigma(\ell)\}, \\
M &:= \O_{K,\Sigma(\ell)}^\times ,\\
G &:= \prod_{\q \in \cP_0}\O_\l^\times, \text{ and }\\
H &:= \prod_{\l \in Q }\O_\l^\times \times \prod_{v \in \Sigma(\ell)}K_v^\times.
\end{align*}
By Remark \ref{largerthansigma}, the map $M/M^p \to G/G^p$ is injective. Therefore
\begin{align*}
\Xset^p(K) = \Hom((G\times H)/M, \boldsymbol{\mu}_p) \too & \Hom(H, \boldsymbol{\mu}_p) \\
& \cong \prod_{\l \in Q }\Hom(\O_\l^\times,\boldsymbol{\mu}_p) \times \prod_{v \in \Sigma(\ell)}\Hom(K_v^\times,\boldsymbol{\mu}_p)
\end{align*}
is surjective by Lemma \ref{KMRlemma}. Since the map is surjective, there exists a $\chi \in \Xset^p(K)$ satisfying
\begin{itemize}
\item
$\chi_{\ell} = \psi_{\ell, e}$,
\item
$\chi_{\q}|_{\O_\l^{\times}} = 1_\q$ for $q \in Q$, \text{ and }
\item
$\chi_{v} = 1_v$ for $v \in \Sigma$,
\end{itemize}
where $\chi_\ell, \chi_\l, \chi_v$ are the restrictions of $\chi$ to $G_{K_\ell}, G_{K_\l}, G_{K_v}$, respectively.
Then in particular, $\chi_\l$ is an unramified character if $\l \in Q$. Note that by Lemma \ref{localzero} the local conditions of two Selmer groups $\Sel_\pi(J^\chi/K)$ and $\Sel_\pi(J/K)$ are the same except at $\ell$, namely, $\alpha_\p(\chi_\p) = \alpha_\p(1_p)$ for $\p \neq \ell$. Therefore, $\Sel_\pi(J^\chi/K) = \Sel_{\pi}(J, \eta_{\ell, e})$, so by Proposition \ref{amazingprime},
$$
d_p(\Sel_\pi(J^\chi/K)) = d_p(\Sel_\pi(J/K)) + 1
$$
\end{proof}
Finally Proposition \ref{superelliptic}, Proposition \ref{hoohoo}, and induction show the following.
\begin{thm}
\label{sec}
Suppose that $K$ is a number field containing $\zeta_p$, and $f \in K[x]$ is a separable polynomial. Let $n = \deg(f)$ and suppose that $p\nmid n$ is an odd prime and $\Gal(f) \cong S_n$. Let $J:= J(C_{p,f})$. Then for any nonnegative integer $r$, there exist infinitely many $p$-twists $J^\chi$ such that $d_p(\Sel_\pi(J^\chi/K)) = r$.
\end{thm}


\section{Canonical quadratic forms}
Until the end of the paper, we are only interested in the hyperelliptic curve case, so we simply write $C_f$, and $\Xset(K)$ (or, $\Xset(K_v))$ to denote the hyperelliptic curve $C_{2,f}$ given by $y^2 = f(x)$, and $\Xset^2(K)$ (or, $\Xset^2(K_v))$, respectively. We denote the Jacobian of $C_{f}$ by $J$. We give a proof of the fact that the two quadratic forms (defined below) on $J[2]$ and $J^\chi[2]$ induced from the Heisenberg groups coincide. This enables us to show a parity relation between two Selmer groups $\Sel_2(J/K)$ and $\Sel_2(J^\chi/K)$ (See Theorem \ref{parity}).

Let $V$ be a $\Ftwo$-vector space. Following \cite{KMR}, we define quadratic forms, metabolic spaces, and Lagrangian (maximal isotropic) subspaces.
\begin{defn}
\label{metabolic}
A {\em quadratic form} on $V$ is a function $q : V \to \Ftwo$ such that
\begin{itemize}
\item
$q(av) = a^2 q(v)$ for every $a \in \Ftwo$ and $v \in V$, \text{ and }
\item
the map $(v,w)_q := q(v+w)-q(v)-q(w)$ is a bilinear form.
\end{itemize}
We say that $X$ is a {\em Lagrangian subspace} or {\em maximal isotropic subspace} of $V$ if $q(X)$ = 0 and $X$ = $X^{\perp}$ in the induced bilinear form.

  A {\em metabolic space} $(V, q)$ is a vector space such that $( , )_q$ is nondegenerate and $V$ contains a {\em Lagrangian subspace}.
\end{defn}

\begin{lem}
\label{numberoflaglangian}
Suppose that $(V, q)$ is a metabolic space such that $d_2(V) = 2n$. Then for
a given Lagrangian subspace $X$ of $V$, there are exactly $2^{n(n-1)/2}$ Lagrangian subspaces that intersect $X$ trivially; i.e.,
$$
|\{ Y: Y\text{ is a Lagrangian subspace such that }Y \cap X = \{0\} \}| = 2^{n(n-1)/2}.
$$
\end{lem}

\begin{proof}
 This is immediate from Proposition 2.6 (b),(c), and (e) in \cite{poonenrain}.
\end{proof}

The most interesting case for our purposes is when $V = H^1(K_v, J[2])$ for local fields $K_v$. In this case, there is a canonical way to construct a quadratic form $q_\H$ using the Heisenberg group defined below (for more general case, see \cite[\S4]{poonenrain}). The associated  bilinear form (Definition \ref{metabolic}) given by such a quadratic form is the same as the pairing \eqref{tld} (see
\cite[Corollary 4.7]{poonenrain}). Then $\alpha_v(1_v)$ is a Lagrangian space by \cite[Proposition 4.9]{poonenrain}, so $(H^1(K_v, J[2]), q_\H)$ is a metabolic space. We explain the construction of $q_\H$ in more detail below. Following \cite[Theorem A.8.1.1]{silverman3} we define a Theta (Weil) divisor.
\begin{defn}
Let $C$ be a smooth projective curve of genus $g \ge 1$. If $j : C \to J(C)$ is an injection, define a
Theta (Weil) divisor (depending on $j$) $\Theta_{J(C), j}$ by
$$
\Theta_{J(C), j} : = j(C) + \cdots + j(C)   \text{  ($g$-1 copies)}.
$$
\end{defn}

\begin{rem}
Our main interest is when $C$ is a hyperelliptic curve $C_{f}$. In such case, we fix an embedding $j : C_{f} \to J$ by sending $(x,y)$ to $[(x,y) - \infty]$. Then the Theta divisor $\Theta_J$ satisfies
$$
[-1]^*\Theta_{J} = \Theta_{J},
$$
since $-[(x,y) - \infty] = [(x,-y) - \infty]$ in $\mathrm{Pic}^0(C_{f}) (\cong J)$.
\end{rem}
Now we define the Heisenberg group for $[2] : J \to J$.
\begin{defn}
\label{quadform}
The Heisenberg group $\mathcal{H}_{J,\Theta}$ is defined by
$$
\mathcal{H}_{J,\Theta_J} := \{(x, g): x \in J[2] \text{, and } g \in \k(J) \text{ such that } \mathrm{div}(g) = 2\tau_x^*(\Theta_J) - 2\Theta_J  \}
$$
where $\tau_x$ is translation by $x$, and $\k(J)$ is the function field of $J$. The group operation is given by $(x, g)(x', g') = (x+x', \tau_{x'}^*(g)g')$.
\end{defn}

\begin{rem}
\label{5.6}
By \cite[Remark 4.5]{poonenrain} and \cite[Corollary A.8.2.3]{silverman3}, we see that Definition \ref{quadform} is a special case of the definition given in the paragraph just before \cite[Proposition 4.6]{poonenrain}.
Let $L$ be a field of characteristic $0$, over which $J$ is defined. There is an exact sequence
\begin{equation}
\label{heisenexact}
1 \to \overline{L}^\times \to \mathcal{H}_{J,\Theta_J} \to J[2] \to 0,
\end{equation}
where the middle maps are given by sending $t$ to ($0$, $t$), and by projection. Then a quadratic form $\q_\H$ is given by the connecting homomorphism
$$
H^1(L, J[2]) \to H^2(L, \overline{L}^\times).
$$
Note that the construction of $q_\H$ is functorial with respect to base extension.
\end{rem}

\begin{defn}
Let $C_{f}$ be a hyperelliptic curve over a local field $K_v$. Let $J$ be the Jacobian of $C_{f}$. Define
$$
\l_{J,v} : H^1(K_v, J[2]) \to H^2(K_v, \overline{K}_v^\times) \cong \Q/\Z
$$
given by the connecting homomrphism of the exact sequence
$$
1 \to \overline{K}_v^\times \to \mathcal{H}_{J,\Theta_J} \to J[2] \to 0.
$$
\end{defn}

\begin{lem}
\label{brauer}
Let $C_{f}$ be a hyperelliptic curve over a number field $K$. Suppose that $x \in H^1(K, J[2])$. Then
$$
\sum_{v} q_{\H,v}(\mathrm{res}_v(x)) = 0,
$$
where $\mathrm{res}_v$ is the restriction map from $H^1(K, J[2])$ to $H^1(K_v, J[2])$.
\end{lem}

\begin{proof}
We have an exact sequence (see \cite[Theorem 8.1.17]{cohomology} for reference)
$$
\xymatrix{
0 \ar[r] & \mathrm{Br}(K) \ar[r] & \bigoplus_v \mathrm{Br}(K_v) \ar^-{\oplus \mathrm{inv}_v}[r] & \Q/\Z \ar[r] & 0.
}
$$
The lemma follows from the functoriality mentioned in Remark \ref{5.6}.
\end{proof}

\begin{lem}
\label{poonen}
Let $\k(J)$ be the function field of $J$. Suppose that $g \in \k(J)$ satisfies $\mathrm{div}(g) = 2\tau_x^*(\Theta_J) - 2\Theta_J$ for some $x \in J[2]$. Then
$g \circ [-1] = g$.
\end{lem}

\begin{proof}
Note that $\mathrm{div}([-1]^*(g)) = [-1]^*(2\tau_x^*(\Theta_J) - 2\Theta_J)
= 2\tau_{-x}^*(\Theta_J) - 2\Theta_J = \mathrm{div}(g)$ since $[-1]^*\Theta_J = \Theta_J$. Hence $[-1]^*g = cg$ for some constant $c$, and $c$ has to be either $1$ or $-1$.

Let $\eta$ be the generic point which corresponds to the divisor $\Theta_J$. Write $\widehat{\O}_{\eta}$ for the completion of the local ring $\O_\eta$ (since $\Theta_J$ is an irreducible divior, $\O_\eta$ is a discrete valuation ring). Then there is an isomorphism
$$
\widehat{\O}_{\eta} \cong k(\Theta)[[t]],
$$
where $k(\Theta)$ is the residue field of $\O_\eta$, and $t$ is a uniformizer. Then $[-1]^* \in \Aut(k(\Theta)[[t]])$ is induced by $[-1]$. Since $[-1]^*$ has order $2$, $[-1]^*$ sends $t$ to $(\pm t) + (\text{higher degree terms})$.  By assumption, $v_{\Theta}(g) = -2$, where $v_{\Theta}$ denote the valuation along $\Theta_J$. Hence if one views $g$ as an element in $k(\Theta)((t))$, it is immediate that $[-1]^*g$ and $g$ have the same leading  term. Therefore $c = 1$, and this completes the proof.
\end{proof}

Let $J'$ be the Jacobian of another hyperelliptic curve and $\lambda : J \to J'$ be an isomorphism over $\overline{K}_v$. By functoriality, the isomorphism $\lambda$ induces an isomorphism $\lambda^*:\mathcal{H}_{J',\Theta_{J'}} \to \mathcal{H}_{J,\Theta_J}$. It is easy to check that the map
$$
\mathrm{Isom}(J, J') \to \mathrm{Isom}(\mathcal{H}_{J',\Theta_{J'}}, \mathcal{H}_{J,\Theta_J})
$$
given by $\lambda \mapsto \lambda^*$ is a $G_{K_v}$-equivariant homomorphism.
Now we show that the induced quadratic forms above are indeed the same for all quadratic twists. The following theorem generalizes \cite[Lemma 5.2]{KMR}.
\begin{thm}
Suppose that $\chi\in \Xset(K_v)$. The canonical isomorphism $J[2] \cong J^\chi[2]$ identifies $\l_{J, v}$ and $\l_{J^\chi, v}$ for every place $v$.
\end{thm}

\begin{proof}
Fix an isomorphism $\lambda: J \to J^\chi$ defined over the field $\overline{K}_v^{\mathrm{Ker}(\chi)}$. For every $\sigma \in G_{K_v}$, we have
$$
\lambda^\sigma = \lambda \circ [\chi(\sigma)] = \lambda \circ [\pm1].
$$
Hence $(\lambda^*)^\sigma = (\lambda^\sigma)^* = [\pm1]^* \circ \lambda^*$.  For all $g$ such that $\mathrm{div}(g) = 2\tau_x^*(\Theta_J) - 2\Theta_J$ for some $x \in J[2]$, we have $[-1]^*g = g$ by Lemma \ref{poonen}. Therefore $(\lambda^*)^\sigma = \lambda^*$ for all $\sigma \in G_{K_v}$, whence $\l_{J,v} = \l_{J^\chi, v}$ since the following diagram commutes.
$$
\xymatrix{
1 \ar[r] & \overline{K}_v^\times \ar[r] & \mathcal{H}_{J,\Theta_J} \ar[r] & J[2] \ar[r] & 0 \\
1 \ar[r] & \overline{K}_v^\times \ar@{=}[u] \ar[r] & \mathcal{H}_{J^\chi, \Theta_{J^\chi}} \ar[r] \ar^-{\simeq}_{\lambda^*}[u] & J^\chi[2] \ar[r] \ar^-{\simeq}[u]& 0.
}
$$
\end{proof}

Combining this theorem with \cite[Theorem 3.9]{KMR}, we get the following.
\begin{thm}
\label{parity}
Suppose that $J$ is the Jacobian of a hyperelliptic curve $C_{f}$ over a number field $K$. Suppose that $\chi \in \Xset(K)$. Then
$$
 \dim_{\Ftwo}\Sel_{2}(J/K) -  \dim_{\Ftwo}\Sel_{2}(J^\chi/K) \equiv \displaystyle{\sum_{v}}h_v(\chi_v) (\textrm{mod } 2),
$$
where $\chi_v$ is the restriction of $\chi$ to $G_{K_v}$ and $h_v$ is given in Definition \ref{hv}.
\end{thm}

\section{$2$-Selmer ranks of hyperelliptic curves}
Let $K$ be a number field and $f\in K[X]$ be a separable polynomial of odd degree ($\ge 3$) such that $\alpha_1, \alpha_2, \cdots \alpha_n$ are the roots of $f$. By an appropriate transformation of $C_f$, we may assume that $\alpha_i$ are algebraic integers. We are mainly interested in the case where $\Gal(f) \cong S_n$ or $A_n$.

Let $\Delta_f$ (:$=\prod_{i < j}(\alpha_i - \alpha_j)^2)$ be the discriminant of the polynomial $f$. Let $\Sigma$ be a set of primes containing all archimedean places, all primes above $2$, and all primes that divide $\Delta_f$ (hence $C_f$, so $J(C_f)$ also, has good reduction at all primes not in $\Sigma$). We enlarge $\Sigma$ so that $\Pic(\O_{K, \Sigma}) = 1$, where $\O_{K, \Sigma}$ is the ring of $\Sigma$-integers, and fix it from now on. Note that
\begin{align*}
\sqrt{\Delta_f} \in \O_{K, \Sigma}^\times & \textrm{ if } \Gal(f) = A_n, \text{ and } \\
\sqrt{\Delta_f} \not\in \O_{K, \Sigma}^\times & \textrm{ if } \Gal(f) = S_n.
\end{align*}

\begin{lem}
\label{deltaf}
\begin{enumerate}
\item
If $\l \in \cP_i$ for some even $i$ and $\chi_\l \in \Xset(K_\l)$, then $\chi_\l(\Delta_f) = 1.$
\item
If $\l \in \cP_i$ for some odd $i$ and $\chi_\l \in \Xset(K_\l)$, then $\chi_\l(\Delta_f) = 1$ if and only if $\chi_\l$ is unramified.
\end{enumerate}
\end{lem}

\begin{proof}
It is well-known that $\sqrt{\Delta_f}$ is fixed exactly by even permutations. The condition $\l \in \cP_i$ is equivalent to $\Frob_\l|_{K(J[2])} \in S_n$ being a product of $i+1$ disjoint cycles by Lemma \ref{dcyc}. Therefore if $i$ is even, then  $\Frob_\l|_{K(J[2])}$ is an even permutation because $n$ is odd, so $\sqrt{\Delta_f} \in K_\l$. In other words, $\Delta_f \in (K_\l^\times)^2$; i.e., $\chi_\l(\Delta_f) = 1$ for all $\chi_\l \in \Xset(K_\l)$. This shows (i). If $i$ is odd, then $\Frob_\l|_{K(J[2])}$ is an odd permutation, so it does not fix  $\sqrt{\Delta_f}$. Hence $\Delta_f \not\in (K_\l^\times)^2$. Therefore by the definition of $\Sigma$ and $\cP_i$(not intersecting $\Sigma$), the discriminant $\Delta_f$ must generate $\O_\l^\times/(\O_\l^\times)^2 \cong \Z/2\Z$, from which (ii) follows.
\end{proof}

\begin{lem}
\label{kernelofthemaptopzero}
Define $\A \subset K^\times/(K^\times)^2$ by
$$
\A := \ker(\O_{K, \Sigma}^\times /(\O_{K, \Sigma}^\times)^2 \to \prod_{\l \in \cP_0}\O_\l^\times/(\O_\l^\times)^2).
$$
Then $\A$ is generated by $\Delta_f$ if $\Gal(f) \cong S_n$, and $\A$ is trivial if $\Gal(f) \cong A_n$.
\end{lem}

\begin{proof}
If $\Gal(f) \cong S_n$, there is only one intermediate field $K(\sqrt{\Delta_f})$ between $K$ and $K(J[2])$ of degree $2$ over $K$. Hence if $\alpha \in \O_{K, \Sigma}^\times /(\O_{K, \Sigma}^\times)^2$ is not equal to $\Delta_f$, then $K(\sqrt{\alpha})$ and $K(J[2])$ are linearly disjoint over $K$. Then by the Chebotarev Density Theorem, there exists a prime $\l$ and $\Frob_\l \in \Gal(K(J[2])K(\sqrt{\alpha})/K)$ such that
\begin{itemize}
\item
$\Frob_\l(\sqrt{\alpha}) = -\sqrt{\alpha}, \text{ and }$
\item
$\Frob_\l|_{K(J[2])} \in \Gal(f) = S_n$ is an $n$-cycle.
\end{itemize}
By Lemma \ref{cycle} and the conditions on $\Frob_\l$, we have $\l \in \cP_0$ and $\sqrt{\alpha} \not\in \O_\l^\times$, so $\alpha \not \in \A$. Lemma \ref{deltaf}(i) with $i=0$ shows that $\Delta_f \in \A$. If $\Gal(f) \cong A_n$, the same argument with $\Delta_f \in (\O_{K, \Sigma}^\times)^2$ shows that $\A$ is trivial.
\end{proof}

For a prime $\l$, we write $\mathrm{res}_\l(\Sel_2^{\l}(J/K))$ for the image of $\Sel_2^\l(J/K)$ of the map
$$
\mathrm{res}_\l : H^1(K, J[2]) \to H^1(K_\l, J[2]).
$$

\begin{lem}
\label{resrelaxedselmer}
The $\Ftwo$-vector space $\mathrm{res}_\l(\Sel_{2}^{\l}(J/K))$ is a Lagrangian subspace in the metabolic space $(H^1(K_\l, J[2]), q_{J, \l}),$ where $q_{J, \l}$ is the quadratic form arising from the Heisenberg group of $J[2]$.
\end{lem}

\begin{proof}
By Lemma \ref{ptd},
$$
d_2(\mathrm{res}_\l(\Sel_{2}^{\l}(J/K))) = \frac{1}{2}d_2(H^1(K_\l, J[2])).
$$
Then \cite[Proposition 4.9]{poonenrain}, \cite[Corollary 4.7]{poonenrain} and Lemma \ref{brauer} show that $\mathrm{res}_\l(\Sel_{2}^{\l}(J/K))$ is a Lagrangian subspace.
\end{proof}

Let $\l$ be a place where $\l\nmid p\infty$ and $J$ has good reduction. Recall that the localization map
$$
\loc_\l: \Sel_2(J/K) \to \alpha_\l(1_\l) \cong J[2]/(\Frob_\l-1)J[2]
$$
is given by evaluating cocycles at $\Frob_\l$. The following two Propositions are the main ingredient of Theorem \ref{selmerrankpm2}.

\begin{lem}
\label{rs}
Suppose that $\Gal(f) \cong A_n$ or $S_n$. For an even number $i$, suppose that $\ell \in \cP_i$ and $\psi_\ell \in \Xset(K_\ell)$. Then, there is a $\chi \in \Xset(K)$ such that
$$\Sel_2(J^\chi/K) = \Sel_2(J, \chi_\ell).$$
\end{lem}

\begin{proof}
Recall that $\Pic(\O_{K,\Sigma}) = 0$, so $\Pic(\O_{K,\Sigma(\ell)}) = 0$. Thus global class field theory shows that
$$
\Xset(K) = \Hom(\iK/K^\times, \pm1) = \textstyle\Hom((\prod_{v \in \Sigma(\ell)}K_v^\times \times
      \prod_{\l\notin\Sigma(\ell)}\O_\l^\times)/\O_{K,\Sigma(\ell)}^\times, \pm1).
$$
Let
\begin{align*}
Q &:= \cP - \{\cP_0 \cup \Sigma(\ell)\}, \\
M &:= \O_{K,\Sigma(\ell)}^\times ,\\
G &:= \prod_{\q \in \cP_0}\O_\l^\times, \text{ and }\\
H &:= \prod_{\l \in Q }\O_\l^\times \times \prod_{v \in \Sigma(\ell)}K_v^\times.
\end{align*}
Define a map $\phi$
\begin{align*}
\phi: \Xset(K) = \Hom((G\times H)/M, \pm1) \too & \Hom(H, \pm1) \\
& \cong \prod_{\l \in Q }\Hom(\O_\l^\times, \pm1) \times \prod_{v \in \Sigma(\ell)}\Hom(K_v^\times, \pm1).
\end{align*}
Then by Lemma \ref{kernelofthemaptopzero} and Lemma \ref{KMRlemma}, $\phi$ is surjective if $\Gal(f) \cong A_n$, and $\mathrm{Im}(\phi)$ is exactly
$\{g \in \Hom(H,\pm1): g(\Delta_f) = 1\}$ if $\Gal(f) \cong S_n$. In either case, for all local characters $\psi_\ell \in \Xset(K_\ell)$, there is a global character $\chi \in \Xset(K)$ such that
\begin{itemize}
\item
$\chi_{\ell} = \psi_{\ell}$,
\item
$\chi_{\q}|_{\O_\l^{\times}} = 1_\q$ for $q \in Q$,
\item
$\chi_{v} = 1_v$ for $v \in \Sigma$
\end{itemize}
by Lemma \ref{deltaf}, where $\chi_\ell, \chi_\l, \chi_v$ are the restrictions of $\chi$ to $G_{K_\ell}, G_{K_\l}, G_{K_v}$, respectively. For example, if $\Gal(f) \cong S_n$, the existence of such a $\chi$ can be seen by Lemma \ref{deltaf}(i). Then by Lemma \ref{localzero}, $\alpha_\p(1_\p) = \alpha_\p(\chi_\p)$ for all places $\p$ except $\ell$.
\end{proof}

\begin{prop}
\label{localtwistdecreasebytwo}
Suppose that $\Gal(f) \cong A_n$ or $S_n$ and suppose that $d_2(\Sel_2(J/K)) \ge 2$. Then there exist infinitely many $\chi \in \Xset(K)$ such that
$$
d_2(\Sel_2(J^\chi/K)) = d_2(\Sel_2(J/K)) - 2.
$$
\end{prop}

\begin{proof}
Decreasing $2$-Selmer rank by $2$ by twisting when $n =3$ and $\Gal(f) \cong A_3$ is done in \cite[Proposition 5.2]{hilbert}. Thus, assume that $C_f$ satisfies $(\ast)$ (Definition \ref{star}). Choose $\ell \in \cP_2$ so that $d_2(\mathrm{Im}(\loc_\ell)) = 2,$ which is poosible by Lemma \ref{3.5} and the Chebotarev Density Theorem. Then $d_2(\Sel_{2,\ell}(J/K)) = d_2(\Sel_2(J/K))-2$. By Lemma \ref{ptd},
$$
d_2(\Sel_2^\ell(J/K)) = d_2(\Sel_{2,\ell}(J/K)) + 2,
$$
whence $\Sel_2^\ell(J/K) = \Sel_2(J/K)$. Taking any ramified character $\psi_\ell \in \Xset_\ram(K_\ell)$, we see that $d_2(\Sel_2(J, \psi_\ell)) = d_2(\Sel_2(J/K)) - 2$. The rest follows from Lemma \ref{rs}.
\end{proof}

\begin{prop}
\label{localtwistincreasebytwo}
Suppose that $\Gal(f) \cong A_n$ or $S_n$. Then there exist infinitely many $\chi \in \Xset(K)$ such that
$$
d_2(\Sel_2(J^\chi/K)) = d_2(\Sel_2(J/K)) + 2.
$$
\end{prop}

\begin{proof}
First assume that $C_f$ satisfies $(\ast)$. Choose $\ell \in \cP_2$ so that $\mathrm{Im}(\loc_\ell) = 0$ and $\Frob_\ell|_{J(K[2])} \in \Gal(f) \subseteq S_n$ is a product of $3$ disjoint cycles of odd lengths. Choosing such an $\ell$ is possible by Lemma \ref{3.5} and the Chebotarev Density Theorem. If $n=3$ and $\Gal(f) \cong A_3$, one can find a sufficiently big field $N$ containing $K(J[2])$ that is Galois over $K$, and $c(\sigma) = 0$ for $\sigma \in G_N$ and $c \in \Sel_2(J/K)$. Then there are infinitely many primes $\ell (\in \cP_2)$ such that $\Frob_\ell|_{\Gal(N/K)} = 1$ by the Chebotarev density theorem.

In either case, we have $\Sel_2(J/K) = \Sel_{2,\ell}(J/K)$. By Lemma \ref{ptd} and Lemma \ref{resrelaxedselmer}.
$$
d_2(\Sel_2^\ell(J/K)) = d_2(\Sel_{2,\ell}(J/K)) + 2
$$
and $\mathrm{res}_\ell(\Sel_2^\ell(J/K))$ is a Lagrangian subspace (Lemma \ref{resrelaxedselmer}) of the metabolic space $(H^1(K_\ell, J[2]),q_{J, \ell})$ that intersects $\alpha_{\ell}(1_{\ell})$ trivially. Let $\Xset_\ram(K_{\ell}) = \{\psi_1, \psi_2 \}$. Then $\alpha_{\ell}(1_{\ell}) \cap \alpha_{\ell}(\psi_1) = \alpha_{\ell}(1_{\ell}) \cap \alpha_{\ell}(\psi_2) = \{ 0 \}$ by Lemma \ref{ramd}, and $\alpha_{\ell}(\psi_1) \cap \alpha_{\ell}(\psi_2) = \{ 0 \}$ by Lemma \ref{rlcd}. By Lemma \ref{numberoflaglangian}, there are exactly $2$ Lagrangian subspaces that intersect $\alpha_{\ell}(1_{\ell})$ trivially, so there exists a $\psi_\ell \in \Xset_\ram(K_\ell)$ such that $\alpha_\ell(\psi_\ell) = \mathrm{res}_\ell(\Sel_2^\ell(J/K))$. Hence it follows that $d_2(\Sel_2(J, \psi_\ell)) = d_2(\Sel_2(J/K)) + 2$. Now Lemma \ref{rs} proves the proposition.
\end{proof}

Finally, Proposition \ref{localtwistdecreasebytwo} and Proposition \ref{localtwistincreasebytwo} show the following by induction.

\begin{thm}
\label{selmerrankpm2}
Suppose that $C_f$ is a hyperelliptic curve over a number field $K$ such that $\Gal(f) \cong A_n$ or $S_n$. Let $J$ denote the Jacobian of $C_f$. Then for all $r \equiv d_2(\Sel_2(J/K)) \text{ } (\text{mod }2)$, there exist infinitely many quadratic characters $\chi \in \Xset(K)$ such that $d_2(\Sel_2(J^\chi/K)) = r$.
\end{thm}

\section{Parity of $2$-Selmer ranks of Jacobians of hyperelliptic curves}
We continue to assume that $J$ is the Jacobian of $C_f$, where $n= \deg(f)$ is odd.

\begin{defn}
\label{7.1}
For every $v \in \Sigma$ and $\chi_v \in \Xset(K_v)$, we define $\omega_v : \Xset(K_v) \to \{ \pm1 \}$ by
$$
\omega_v(\chi_v):= (-1)^{h_v(\chi_v)}\chi_v(\Delta_f).
$$
Define
$$
\delta_v:= \frac{1}{|\Xset(K_v)|}\sum_{\chi \in \Xset(K_v)}\omega_v(\chi) \text {  } \text{     and     } \text {  }\delta:= (-1)^{d_2(\Sel_2(J/K))}\prod_{v \in \Sigma}\delta_v.
$$
\end{defn}

\begin{defn}
\label{7.2}
Define a function $\Xset(K) \to \Z_{>0}$ by
$$
\Vert \chi \Vert := \text{max}\{\N(\l): \chi \text{ is ramified at }\l\},
$$
where $\N(\l)$ is the index of the residue field of $K_\l$. If $X>0$, let $\Xset(K,X) \subset \Xset(K)$ be the subgroup
$$
\Xset(K,X) := \{\chi \in \Xset(K):\Vert \chi \Vert < X\}.
$$
\end{defn}

The following proposition is in fact the same as \cite[Proposition 7.2]{KMR} in a slightly more general setting (hyperelliptic curves).
For $\chi \in \Xset(K)$, let $$r(\chi) := d_2(\Sel_2(J^\chi/K))$$ and $\chi_v$ be the restriction of $\chi$ to $G_{K_v}$.

\begin{prop}
\label{congruence}
Suppose that $\chi \in \Xset(K).$ Then
$$
r(\chi) \equiv r(1_K) \text{ } (\text{mod }2) \Longleftrightarrow \prod_{v \in \Sigma} \omega_v(\chi_v) = 1.
$$
\end{prop}

\begin{proof}
Let $\theta$ be a (formal) product of primes not in $\Sigma$ such that $\chi$ is ramified exactly at the primes which divide $\theta$. Then by Lemma \ref{deltaf},
$$
\prod_{\l \not\in \Sigma} \chi_\l(\Delta_f) = (-1)^{|\{\l:\l\in\cP_i \text{ for odd $i$ and $\l|\theta$}\}|}.
$$
 Note that for $\l|\theta$, we have $(-1)^{h_\l(\chi_\l)} = \chi_\l(\Delta_f)$ by Lemma \ref{ramd}, Lemma \ref{deltaf}, and Remark \ref{localdim}. Therefore Theorem \ref{parity} and Lemma \ref{localzero} show that
\begin{align*}
r(\chi) \equiv r(1_K) \text{ (mod }2) & \Longleftrightarrow (-1)^{\Sigma_{v}h_v(\chi_v)} = 1 \\
& \Longleftrightarrow \prod_{v \in \Sigma} \omega_v(\chi_v)\chi_v(\Delta_f) \prod_{v \not\in \Sigma} \chi_v(\Delta_f) = 1.
\end{align*}
Clearly $\prod_v \chi_v(\Delta_f) = 1$, so this completes the proof.
\end{proof}

The following theorem is remarkable theorem due to Klagsbrun, Mazur, and Rubin (\cite[Theorem 7.6]{KMR}) (for the elliptic curve case).
\begin{thm}
\label{paritydensity}
For all sufficiently large $X$,
$$
\frac{|\{\chi \in \Xset(K, X):d_2(\Sel_2(J^\chi/K)) \text{ is even }    \}|}{|\Xset(K, X)|} = \frac{1 + \delta}{2}.
$$
\end{thm}

\begin{proof}
See the proof of Theorem 7.6 in \cite{KMR}. Without difficulty, one can see \cite[Theorem 7.6]{KMR} can be extended to the hyperelliptic curve case.
\end{proof}

\begin{prop}
\label{delta0}
Suppose that $K$ has a real embedding $K \hookrightarrow K_{v_0}$. Then
\[\delta_{v_0} =
\begin{cases}
1 \text{ if } n \equiv 1 \text{ } (\text{mod } 4)            \\
0 \text{ if } n \equiv 3 \text{ } (\text{mod } 4).
\end{cases}
\]
\end{prop}

\begin{proof}
Let $\eta$ be the sign character in $\Xset(K_{v_0}) = \Hom(K_{v_0}^\times, \pm1)$ sending negative numbers to $-1$.
Suppose that $f$ has $2k_1-1$ real roots and $2k_2$ complex roots so that $2k_1 + 2k_2 - 1 = n$. Let $\beta_1, \overline{\beta_1}, \beta_2, \overline{\beta_2}, \cdots \beta_{k_2}, \overline{\beta_{k_2}}$ denote the complex roots of $f$, where $\overline{\beta_i}$ is the complex conjugate of $\beta_i$.
Then by an appropriate rearrangement of the roots, we get
$$
\Delta_f = \prod_{i<j}(\alpha_i -\alpha_j)^2 = c\overline{c}\prod_{1 \le i \le k_2}(\beta_i -\overline{\beta_i})^2
$$
for some c. Hence $\eta(\Delta_f) = (-1)^{k_2}$. By Lemma \ref{reallocal}, we deduce that
$$
(-1)^{h_{v_0}(\eta)} = (-1)^{k_1-1}.
$$
Therefore $\delta_{v_0} = 1/2(1 + \omega_{v_0}(\eta)) = 1/2(1 + (-1)^{k_1+k_2-1})$, so the proposition follows from the equality $2k_1 + 2k_2 - 1 = n$.
\end{proof}
As an easy application, we have
\begin{cor}
\label{application}
Suppose that $n \equiv 3 \text{ } (\text{mod } 4),$ and $K$ has a real embedding. Then for all sufficiently large $X$, we have
$$
|\{\chi \in \Xset(K, X): r(\chi) \textrm{ is even }    \}| = |\{\chi \in \Xset(K, X): r(\chi) \textrm{ is odd }    \}|=\frac{|\Xset(K,X)|}{2}.
$$
\end{cor}

\begin{rem}
If $n \equiv 1 \text{ } (\text{mod } 4)$ or if $K$ has no real embedding, we have to know values of $\delta_v$ other than $\delta_{v_0}$ to compute an accurate density. In fact, the ``disparity constant" $\delta$ may not be zero. We display an example in next section.
\end{rem}

Theorem \ref{selmerrankpm2} and Theorem \ref{paritydensity} show the following. \begin{thm}
Suppose that $C_f$ is a hyperelliptic curve defined over a number field $K$ such that $n= \deg(f)$ is odd. Suppose that $\Gal(f) \cong S_n$ or $A_n$, and the disparity constant $\delta$ is neither $-1$ nor $1$. Then for every $r \geqq 0$, the Jacobian $J$ of $C_{f}$ has infinitely many quadratic twists $J^\chi$ such that   $d_2(\Sel_2(J^\chi/K))$ = $r$.
\end{thm}

\begin{cor}
\label{hec}
Suppose that $C_f$ is a hyperelliptic curve defined over a number field $K$. Let $n = \deg(f)$, and suppose that $n \equiv 3 (\textrm{mod }4)$ and $\Gal(f) \cong S_n$ or $A_n$. Suppose further that $K$ has a real embedding. Then for every $r \geqq 0$, the Jacobian $J$ of $C_{f}$ has infinitely many quadratic twists $J^\chi$ such that   $d_2(\Sel_2(J^\chi/K))$ = $r$.
\end{cor}



\section{Examples}
In this section, we give an explicit example of a hyperelliptic curve $C_h/\Q$ such that $d_2(\Sel_2(J^\chi/\Q))$ has constant parity for all quadratic twists $J^\chi$ ($J$:= the Jacobian of $C_h$), so it defies \eqref{50percent} in the introduction. More precisely, the main result of this section is the following.

\begin{prop}
\label{heng}
Suppose that $C_h$ is a hyperelliptic curve over $\Q$ whose affine model is
$$
y^2 = h(x) := -273(6x+1)(91x^2+54x+9)(100x^2+60x+1)
$$
Let $J$ denote the Jacobian of $C_h$. Then $d_2(\Sel_2(J^\chi/\Q))$ is even for all quadratic twists $J^\chi$ of $J$.
\end{prop}

Let $E_f$ be an elliptic curve labelled 1440D1 in \cite{table}:
$$
y^2 = f(x) := x^3 -273x +1672.
$$
Then
$$
E_f[2] = \{\infty, (-19,0),(8,0),(11,0) \}.
$$
Define an isomorphism $\psi: E_f[2] \to E_f[2]$ by sending $(\alpha_i,0)$ to $(\beta_i,0)$, where
$$
\alpha_1 = -19, \alpha_2 = 8, \alpha_3 =11, \beta_1 = 8 , \beta_2 = 11, \text{ and }  \beta_3 = -19.
$$
Clearly, $\psi$ does not come from an isomorphism $E_f \to E_f$ since $E_f$ does not have complex multiplication (the $j$-invariant of $E_f$ is not an integer).

Proposition 4 in \cite{example} shows that the Jacobian of the curve defined by $y^2 = h(x)$ where
$$
h(x) = -(-810Ax^2 + 81B)(81Ax^2 -90B)(-90Ax^2 -810B)
$$
is isomorphic to the quotient of $E_f \times E_f$ by the graph of $\psi$. The constant $A$ and $B$ are as in Proposition 4 in \cite{example}, and one can see $A = 1990170 = -B$ by simple algebra. Then by a rational transformation of $y^2 = h(x)$ by
$$
x = \frac{3x'+1}{x'},  y = \frac{cy'}{x'^3},
$$
where $c = 2^2 \times 3^{14} \times 5^2 \times 7 \times 13$,
we get
$$
y'^2 = -273(6x' + 1)(91x'^2 + 54x' +9)(100x'^2 + 60x' + 1).
$$
By abuse of notation, let
$$
h(x) := -273(6x + 1)(91x^2 + 54x +9)(100x^2 + 60x + 1).
$$
Then the above observation shows that $J$ is isogenous to $E_f \times E_f$ over $\Q$.

\begin{defn}
Let $A$ be an abelian variety over a number field $K$. Define
$$
\Sel_n(A/K):= \{x \in H^1(K, A[n]):\mathrm{res}_v(x) \in \mathrm{Im}(i_v) \text{ for all places $v$}\},
$$
where $\mathrm{res}_v$ is the restriction map
$$\mathrm{res}_v: H^1(K, A[n]) \to H^1(K_v, A[n])$$
and $i_v$ is the Kummer map $i_v : A(K_v)/n A(K_v) \to H^1(K_v, A[n])$. If $p$ is a prime, we define $\Sel_{p^\infty}(A/K)$ to be the direct limit of the Selmer groups $\Sel_{p^k}(A/K)$.
\end{defn}

\begin{lem}
\label{lastlemma}
Suppose that $C_h$ and $E_f$ are as above. Then for any $\chi \in \Xset(\Q)$,
$$
d_2(\Sel_2(J^\chi/\Q)) \equiv d_2(J(\Q)[2]) \text{ }(\text{mod  } 2)
$$
\end{lem}

\begin{proof}
Since $J$ and $E_f \times E_f$ are isogenous over $\Q$, the induced map
$$
\Sel_{2^\infty}(J/\Q) \to \Sel_{2^\infty}((E_f \times E_f)/\Q)
$$
has finite kernel and cokernel. Hence
$$
\mathrm{corank}_{\Z_2}(J/\Q) = \mathrm{corank}_{\Z_2}((E_f \times E_f)/\Q),
$$
so $\mathrm{corank}_{\Z_2}(J/\Q)$ is even. In a similar way, one can see that $\mathrm{corank}_{\Z_2}(J^\chi/\Q)$ is even for all quadratic twists $J^\chi$.
We have the following two exact sequences:
$$
\xymatrix@R=5pt@C=15pt{
0 \ar[r] & J(\Q) \otimes \Q_2/\Z_2 \ar[r] & \Sel_{2^\infty}(J/\Q) \ar[r] & \Sh[2^\infty] \ar[r] & 0, \text{ and }    \\
0 \ar[r] &  J(\Q)/2J(\Q) \ar[r] & \Sel_2(J/\Q) \ar[r] & \Sh[2] \ar[r] & 0,
}
$$
where the group $\Sh$ is the Shafarevich-Tate group of $J/\Q$. From the above exact sequences, we see that
\begin{align*}
d_2(\Sel_2(J/\Q)) & = \mathrm{rk}(J(\Q)) + d_2(J(\Q)[2]) + d_2(\Sh_{\mathrm{div}}[2]) + d_2(\Sh/\Sh_{\mathrm{div}}[2]) \\
& = \mathrm{corank}_{\Z_2}(J/\Q) + d_2(\Sh/\Sh_{\mathrm{div}}[2]) + d_2(J(\Q)[2]) \\
& \equiv d_2(J(\Q)[2]) \text{ (mod } 2),
\end{align*}
where the last congruence holds by the following. Note that $C_h$ has a rational point $\infty$, so the ($K$-rational) theta divisor given by $j: C_h \to J$ sending $P$ to $[P - \infty]$ produces a principal polarization. See Section $A.8.2$ of \cite{silverman3} for more details. Then the congruence follows from the following two general facts.
\begin{enumerate}
\item
If $A$ is an abelian variety over a number field $K$ that has a principal polarization coming from a $K$-rational (Weil) divisor, then there is a paring
$$
\Sh_{A/K} \times \Sh_{A/K} \to \Q/\Z,
$$
that is alternating and nondegenerate after division by maximal divisible subgroup.
\item
If there is a finite abelian group $B$ with an alternating non-degenerate pairing
$$
B \times B \to \Q/\Z,
$$
then $d_2(B[2])$ is even.
\end{enumerate}
Similarly, one can see
$$
\dim_{\Ftwo}(\Sel_2(J^\chi/\Q)) \equiv \dim_{\Ftwo}(J^\chi(\Q)[2]) \text{ (mod } 2)
$$
for all quadratic twists $J^\chi$. Then the lemma follows from Proposition \ref{canonicaliso}.
 \end{proof}

\begin{proof}[Proof of Proposition 8.1]
It is easy to see $d_2(J(\Q)[2]) =2$ by Lemma \ref{cycle}. Then Lemma \ref{lastlemma} completes the proof.
\end{proof}

We show one more example in the following proposition.
\begin{prop}
\label{heng2}
Let $J(C_g)$ be the Jacobian of hyperelliptic curve $C_g$ given by
$$
y^2 = g(x) = (2x+1)(3x^2+4x+2)(3x^2+2x+1).
$$
Then $d_2(\Sel_2(J(C_g)^\chi/\Q)$ is even for all $\chi \in \Xset(K)$.
\end{prop}

\begin{proof}
Let $E'$ be the elliptic curve $y^2 = x^3 - x$. Let
$$
\alpha_1 = 1, \alpha_2 = -1, \alpha_3 = 0, \beta_1 = -1, \beta_2 = 0, \text{ and }, \beta_3 = 1.
$$
Then one can proceed exactly in the same way as above to get a hyperelliptic curve $C_{g'}$ given by $y^2 = g'(x)$ for some $g' \in \Q[X]$, whose Jacobian is isogenous to $E' \times E'$ and is a quadratic twist of $J(C_g)$. It is easy to see $d_2(J(C_g)(\Q)[2]) = 2$ by Lemma \ref{cycle}. Then the rest follows from Lemma \ref{lastlemma}.
\end{proof}

\section*{Acknowledgements}
The author is very grateful to his advisor Karl Rubin for sharing his deep ideas and insight into the subject. Also the author wants to thank to Bjorn Poonen, and Heng Su for providing a simpler proof of Lemma \ref{poonen}, and for programming used in the discovery of the example in Proposition \ref{heng}. He extends his thanks to the referees for previous and current versions of this article for helpful suggestions, as well as to Chan-Ho Kim and Wan Lee for comments. He is also grateful to Dennis Eichhorn for proofreading.

\bibliographystyle{abbrv}
\bibliography{mjrefe}

\end{document}